\documentclass[11pt]{article}

\usepackage[top=1.0in, bottom=1.0in, left=1.0in, right=1.0in]{geometry}
\usepackage{amsfonts, amssymb, amsthm, amsmath}
\usepackage{booktabs, threeparttable, multirow}
\usepackage{graphicx}
\usepackage{xcolor}
\usepackage{titlesec} 
\usepackage{caption}[2007/12/23]
\usepackage{afterpage}
\usepackage{algorithm}
\usepackage{algpseudocode}
\usepackage{setspace} 
\usepackage{authblk} 
\usepackage{todonotes}
\usepackage{appendix}
\usepackage{enumerate}
\usepackage{tabularx}
\usepackage{comment}

\makeatletter
\newcommand{\multiline}[1]{%
  \begin{tabularx}{\dimexpr\linewidth-\ALG@thistlm}[t]{@{}X@{}}
    #1
  \end{tabularx}
}
\makeatother

\newcommand{\algorithmicelsif}{\textbf{else if}}
\algdef{C}[IF]{IF}{NoThenElseIf}[1]{\algorithmicelsif\ #1}
\usepackage{hyperref}
\hypersetup{
    colorlinks=true,        
    linkcolor=blue,         
    citecolor=blue,         
    urlcolor=blue,          
    pdfborder={0 0 0}       
}

\usepackage[backend=biber,style=apa]{biblatex} 
\DeclareLanguageMapping{english}{english-apa} 
\AtEveryBibitem{\clearfield{url}} 
\AtEveryBibitem{\clearfield{doi}} 

\newtheorem{definition}{Definition}
\newtheorem{property}{Property}
\newtheorem{proposition}{Proposition}
\newtheorem{theorem}{Theorem}
\newtheorem{lemma}{Lemma}
\newtheorem{corollary}{Corollary}

\makeatletter
\def\BState{\State\hskip-\ALG@thistlm}
\makeatother

\title{\textbf{\LARGE \textcolor{darkblue}{Optimizing Railcar Movements to Create Outbound Trains in a Freight Railyard}}}

\author[a]{Ruonan Zhao}
\author[a]{Joseph Geunes\thanks{Corresponding author.\newline Email: geunes@tamu.edu}}
\author[b,a]{Xiaofeng Nie}

\affil[a]{Wm Michael Barnes '64 Department of Industrial and Systems Engineering, 
Texas A\&M University}
\affil[b]{Department of Engineering Technology and Industrial Distribution, 
Texas A\&M University}

\date{}

\definecolor{darkblue}{RGB}{0, 51, 102} 

\titleformat{\section}
  {\color{darkblue}\normalfont\Large\bfseries} 
  {\thesection}{1em}{} 

\titleformat{\subsection}
  {\color{darkblue}\normalfont\large\bfseries} 
  {\thesubsection}{1em}{} 

\titleformat{\subsubsection}
  {\color{darkblue}\normalfont\normalsize\bfseries} 
  {\thesubsubsection}{1em}{} 

\begin{document}
\maketitle

\begin{abstract}
A typical freight railyard at a manufacturing facility contains multiple tracks used for storage, classification, and outbound train assembly. Individual railcar storage locations on classification tracks are often determined before knowledge of their destination locations is known, giving rise to railcar shunting or switching problems, which require retrieving subsets of cars distributed throughout the yard to assemble outbound trains. To address this combinatorially challenging problem class, we propose a large-scale mixed-integer programming model that tracks railcar movements and corresponding costs over a finite planning horizon. The model permits simultaneous movement of multiple car groups via a locomotive and seeks to minimize repositioning costs. We also provide a dynamic programming formulation of the problem, demonstrate the $\mathcal{NP}$-hardness of the corresponding optimization problem, and present an adaptive railcar grouping dynamic programming (ARG-DP) heuristic, which groups railcars with common destinations for efficient moves. Average results from a series of numerical experiments demonstrate the efficiency and quality of the ARG-DP algorithm in both simulated yards and a real yard. On average, across 60 test cases of simulated yards, the ARG-DP algorithm obtains solutions 355 times faster than solving the mixed-integer programming model using a commercial solver, while finding an optimal solution in $60\%$ of the instances and maintaining an average optimality gap of 6.65\%. In 10 cases based on the Gaia railyard in Portugal, the ARG-DP algorithm achieves solutions 229 times faster on average, finding an optimal solution in $50\%$ of the instances with an average optimality gap of 6.90\%.

\end{abstract} \hspace{10pt}

\affil[*]{Corresponding author: Joseph Geunes. Email: joseph@tamu.edu}

\section{Introduction}
Rail transport plays a significant role in various industries, including manufacturing, agriculture, and logistics, offering a cost-effective and environmentally sustainable alternative to road transport. According to the \textcite{Railroads2020overview}, total U.S.\ freight shipments are expected to increase from approximately 18.6 billion tons in 2018 to 24.1 billion tons by 2040 as forecasted by the U.S. Federal Highway Administration, representing a 30\% increase.  Consequently, efficient movement of freight cars within a transportation network is vital for optimizing operational costs, minimizing transit times, and ensuring the smooth operation of supply chains. 

The general layout of a railyard includes receiving yards, classification yards, and departure yards. Receiving yards and departure yards function as places to receive inbound trains and dispatch outbound trains, respectively. Classification yards play an important role in connecting receiving yards and departure yards by assisting in sorting cars based on their destinations. Within each yard, the tracks can be of two types: one-sided available or two-sided available. One-sided tracks can be considered as stack structures, where railcars are added and removed from only one end, leading to a last-in-first-out (LIFO) retrieval order. In contrast, two-sided tracks function like queue structures, allowing railcars to be added from one end and removed from the opposite end, following a first-in-first-out (FIFO) order \parencite{knuth1968}. In this paper, we focus on one-sided (LIFO) systems, more commonly arising in flat railyards connected to manufacturing facilities, which receive inbound goods from suppliers and assemble outbound trains for delivery to customers.

Freight yards are classified into three main types: flat yards, gravity yards, and hump yards, each utilizing different methods for sorting and organizing railcars. In contrast to gravity yards and hump yards, flat yards do not feature a hump and rely only on locomotives to move railcars. The detailed process of using a locomotive to move railcars is illustrated in Table \ref{tab:steps}. This entire sequence, comprising Steps 1 through 3, is referred to as a shunting move. 

\begin{table}[h!]
\centering
\caption{Steps involved in a shunting move.}\label{tab:steps}
\begin{tabular}{|c|l|}
\hline
Step & Activity \\
\hline
1 & A locomotive is attached to a target individual railcar or a set of railcars, \\
& which are then decoupled from the rest of the train \\
2 & The locomotive actively pulls the railcars out via a ladder track\footnotemark and pushes \\
& them onto some other designated track(s)\\
3 &  The locomotive is decoupled from the railcars, making it available for the \\
& next shunting move\\
\hline
\end{tabular}
\end{table}
\footnotetext{A ladder track is a segment of track that leads to a sequence of multiple classification or departure tracks.}

A crucial shunting planning task in such railyards lies in determining a sequence of railcar moves via a locomotive in order to assemble trains destined for designated locations as efficiently as possible. In current practice, yard planners often employ a manual process in which they plan out sequences of moves based on knowledge gained from experience in solving such problems. This reliance on manual planning not only consumes significant time and resources but also limits the potential for optimizing operations costs in complex and dynamic railyard environments. Additionally, the arrangement of cars across multiple tracks results in a combinatorial explosion of possible movements, further complicating the shunting process. 

To address these challenges, this paper undertakes a comprehensive exploration of a particular type of shunting problem that we refer to as the railcar shunting problem (RSP), beginning with a mathematical definition of the problem. We propose two exact models with an objective of minimizing the shunting cost incurred in relocating specific railcars to their designated outbound tracks. First, we propose a mixed-integer programming (MIP) model that treats sets of consecutive railcars sharing the same destination as unified car groups, ensuring that these groups are moved together without being split. This model also facilitates the simultaneous movement of multiple such car groups, thereby minimizing repositioning costs and enhancing operational efficiency. Second, we introduce an alternative formulation employing a dynamic programming (DP) method, which forms the basis for our heuristic methodology. By constructing a so-called conflict graph associated with the rail network, we establish the $\mathcal{NP}$-completeness of the recognition version of the RSP. We then develop an adaptive railcar grouping dynamic programming (ARG-DP) heuristic algorithm, which contains three key steps: preprocessing, graph construction, and DP. 

The MIP model proposed in this paper has various applications in practice. For example, in container terminals, cranes organize containers in storage blocks that form vertical stacks, typically following a LIFO pattern as described by \textcite{containerstack}. When specific containers must be retrieved to load onto a vessel or truck or for maintenance, the crane operator must determine the sequence of moves, which corresponds to the unloading problem discussed by \textcite{lehnfeld2014loading}. Another practical application occurs in the steel industry, where steel slabs must be retrieved from a slab yard according to a scheduled sequence \parencite{steelshunting}. 

The RSP may arise as a subproblem within larger railway optimization tasks. For example, in integrated yard management, railcars from multiple inbound trains must be reorganized to form outbound trains on time \parencite{ALIAKBARI2024950}. When the RSP arises as a subproblem within a larger railway planning problem, the ability to efficiently solve this subproblem—which is the primary focus of this paper—becomes essential for addressing the overall problem within acceptable computational time.

The contributions of this paper are as follows.

\begin{itemize}
    \item \textbf{Development of a Comprehensive Planning Framework}: We develop a comprehensive planning framework to model the RSP, addressing the complexities inherent in freight railyard operations for a manufacturer, with an objective of minimizing total shunting cost.
    
    \item \textbf{MIP and DP Models}: We provide a general MIP model that enables the simultaneous movement of any number of consecutive railcars with the objective of optimizing shunting costs. We also propose a dynamic programming method tailored to the RSP, characterizing the states, actions, and Bellman equation for this problem.
    
    \item \textbf{Proof of $\mathcal{NP}$-Completeness}: We prove the $\mathcal{NP}$-completeness of the recognition version of the RSP by relating it to the stack sorting problem as described by \textcite{KLstacksorting}.  To the best of our knowledge, this is the first characterization of the complexity of block relocation and related stack sorting problems with batch moves (see., e.g., \cite{zhang2016tree}, for a definition of the block relocation problem with batch moves).
    
    \item \textbf{ARG-DP Heuristic Algorithm}: We develop a heuristic that efficiently solves the RSP by approximately solving the DP formulation.
    
    
    \item \textbf{Numerical Study}: An extensive numerical study demonstrates the efficiency and quality of the proposed ARG-DP heuristic, showcasing its ability to achieve solutions more than 355 times faster than the MIP model, while finding an optimal solution in $60\%$ of the problem instances tested, and maintaining an average optimality gap of 6.65\% for 60 cases of simulated yards. Additionally, in 10 cases based on the Gaia yard structure in Portugal, discussed in \textcite{Gaiayard}, the ARG-DP algorithm achieves solutions 229 times faster on average, while finding an optimal solution for half of the problem instances tested, with an average optimality gap of 6.90\%.
\end{itemize}

The remainder of this paper is organized as follows. In Section \hyperref[sec:lr]{2}, we discuss closely related literature. Section \hyperref[sec:formulation]{3} provides a formal definition of the RSP. In Section \hyperref[sec:model]{4}, we introduce MIP and DP models for the RSP. This is followed by a proof of $\mathcal{NP}$-completeness in Section \hyperref[sec:complexity]{5}. In Section \hyperref[sec:ARG]{6}, we present the ARG-DP heuristic algorithm. Section \hyperref[sec:results]{7} presents computational results from applying our solution methods to solve the RSP both in  simulated yards and using the structure of the Gaia yard in Portugal, detailing its computational results and efficiency. Section \hyperref[sec:conclusion]{8} concludes the paper and provides recommendations for future research.

\section{Literature Review}
\label{sec:lr}
Shunting operations have been extensively studied by researchers and are sometimes referred to as sorting, marshalling, or switching in the literature. Additionally, railcars may be referred to as cars, trams, or other related terms in various studies. Tracks accessible from only one side can be considered analogous to stacks, as they share the  LIFO property. The literature on stack sorting includes works of \textcite{knuth1968}, who first introduced the idea of stack sorting using railway language and characterized permutations that can be sorted using $k$ stacks in series. \textcite{tan} extends these ideas to sorting with acyclic networks of stacks and queues. \textcite{lehnfeld2014loading} provide a nice overview of past work in the area of stack loading and unloading, highlighting its application in various contexts including shipping, rail, and steel stacking.  We next discuss related literature in rail contexts under three objectives: optimizing the number of classification tracks, reducing the number of shunting movements, and minimizing shunting cost or time.

One prominent research stream focuses on minimizing the number of classification tracks required—an issue closely related to managing shunting operations that transfer railcars from inbound tracks to classification tracks. \textcite{trainmarshalling} introduce the train marshalling problem, wherein an incoming train arrives in a specific order, and its cars have different destinations. The objective is to rearrange the cars onto $k$ auxiliary tracks while minimizing the number of auxiliary tracks $k$. Upon completion, all cars from one auxiliary track are positioned at the beginning of the rearranged train, followed by the cars from another track, and so forth. Their study primarily establishes the $\mathcal{NP}$-completeness of the train marshalling problem and subsequently provides an upper bound on the required number of auxiliary tracks. \textcite{RollingStock} present an algorithm to address shunting operations in an actual hump yard, with the same goal of minimizing the number of classification tracks required.

Another key focus of the literature is on optimizing the number of shunting movements. \textcite{schduleinthemorning} introduce a problem of ``scheduling trams in the morning,'' which involves a set of tram cars of different types arranged on separate tracks and a specified departure order for trams of particular types. The objective is to determine whether an assignment sequence of trams to the departure track exists that allows all trams to depart in the required order without any shunting movements. They address the problem using a DP method. \textcite{handbook2018} provide an overview of freight shunting from two perspectives: single-stage classification and multi-stage classification. The key difference is that multi-stage classification permits the movement of railcars between classification tracks, whereas single-stage classification does not. Furthermore, single-stage classification aims to minimize the number of classification tracks, while multi-stage classification focuses on minimizing the number of shunting movements. \textcite{PSPACECompleteStacking} propose a greedy approach for minimizing the number of movements required to relocate items to target tracks, while not allowing movements of multiple cars at the same time. 

Additional studies focus on minimizing shunting cost or distance, which are more closely related to our work.  \textcite{Lubbecke2005} analyze the allocation of railcars within in-plant rail networks. They group neighboring tracks into regions and address two primary sub-problems: determining how many railcars of each type should be supplied from each region for each request, and retrieving the requested railcars by region. The second sub-problem involves shunting operations, but their approach only considers a single outbound track and assumes that retrieval costs increase linearly with the position (depth) of the railcar on that track. In other words, retrieving a car from the sixth position is six times more expensive than retrieving a car from the first position. Furthermore, they do not consider railcar movements between different classification tracks, focusing solely on transfers from a classification track to the outbound track. \textcite{singleengine2015} examine the railcar movement problem using a single locomotive. They present an MIP model that minimizes locomotive travel distance while assembling a train composed of railcars destined for the same location. Their approach assumes that all railcars of interest on any given track are ordered consecutively and are not scattered throughout the track. 
In our work, we do not make this assumption; instead, we consider railcar groups, where groups with different destinations may be arranged arbitrarily along the tracks. Moreover, while their model produces a single train for a given destination, our model offers the flexibility to form multiple trains, depending on the number of destination tracks available—for example, forming four trains if there are four destination tracks. 

\textcite{flatyardRetrieving} consider the retrieval of a specified number of railcars of certain types with the objective of minimizing retrieval costs, which appears to be the most closely related work to ours in the literature to the best of our knowledge. They formulate the problem as an MIP model. In their approach, retrieval costs depend only on whether a car is at the head of a track, and they assume these costs are identical across all classification tracks. In contrast, our model does not make such assumptions; our operational costs per shunting move depend on the distance between tracks and vary accordingly. Furthermore, their model considers only a single outbound track, whereas our model offers the flexibility of multiple outbound tracks. Additionally, they do not account for the sequence of railcar movements, instead treating the problem as a supply-demand scenario aimed at extracting the total required number of specific types of cars. However, our model not only considers the extraction of specific railcars but also the sequence in which railcars are extracted, thereby minimizing the total operational cost.

A related body of work considers the so-called \emph{block relocation} problem, in which a subset of items must be retrieved from a set of stacks of items (see, e.g., \cite{caserta2012mathematical}).  In the standard block relocation problem, items can only be removed from the top of a stack, and are either moved to the top of another stack or removed from the top of a stack. A recent survey \parencite{lersteau2022survey} characterizes work in this area and its relation to railcar sorting problems. As with many stack sorting problem classes, the block relocation problem assumes that items may be moved from the top of a stack one-at-a-time. A notable exception can be found in \textcite{zhang2016tree}, who define the block relocation problem with batch moves, where a group of items may be moved from the top of a stack. As noted by \textcite{lu2020study}, research on the version of the problem with batch moves is extremely limited and provides a promising area for further analysis.  As we later show, we contribute to this stream of literature by demonstrating the $\mathcal{NP}$-hardness of the block relocation problem and other stack sorting problems with batch moves.

Although extensive research has been conducted on shunting operations, the specific movement of freight railcars from classification tracks to departure tracks by locomotives within a typical yard has received limited attention. This issue remains an open question, as noted by \textcite{shuntingreview}. To the best of our knowledge, no exact shunting approaches aiming to minimize the cost of moving items to target tracks and allowing for multiple simultaneous car movements have been suggested, and this is the first time an exact MIP model has been constructed. The power of this model is that it offers considerable versatility, enabling the movement of any number of cars simultaneously between classification tracks or from classification tracks to departure tracks.

\section{Problem Definition}
\label{sec:formulation}
This section describes the RSP in detail. The goal of the RSP is to assist decision-makers in determining an optimal sequence of railcar movements within a single railyard, utilizing a locomotive for all relocations. A railyard comprises various railcars, some assigned to specific departure tracks and others without designated destinations. The layout of a flat railyard includes ladder tracks, classification tracks, and departure tracks (see Figure \ref{systemexample}). 

We focus on one-sided available tracks, where railcars can only be accessed from the \textit{switch end} of each track, while the opposite side is known as the \textit{dead end}. The switch end connects to the ladder tracks, facilitating movement between tracks. Within the railyard, railcars are distributed across classification tracks; those with assigned destinations need to be moved to their respective departure tracks. The operational cost per shunting move is determined by the distance the locomotive travels between track segments and our approach will assume that this cost is fixed for any pair of tracks and independent of the number of cars moved between the tracks. Additionally, we assume that only one locomotive operates at a time within the railyard.

Figure \ref{systemexample} illustrates a sample railyard layout with three departure tracks, three classification tracks, two ladder tracks, and eight railcars. In the figure, railcars are depicted as rectangles, with the black railcar representing the locomotive. Classification tracks are denoted by black horizontal lines, while departure tracks are color-coded (green, blue, and purple) to indicate different destinations. Additionally, each railcar is also color-coded to correspond with its designated departure track. White railcars, however, do not have destinations, and can be assigned or moved to any of the classification tracks. In this example, tracks are interconnected on the left side by ladder tracks, allowing railcars to be accessed exclusively from this switch end and subsequently moved to any other track as needed.

In a typical freight railyard, numerous railcars are positioned on classification tracks. Instead of managing individual railcars, we designate multiple adjacent railcars with the same departure track as a unified railcar \emph{group}, which will remain coupled throughout shunting operations. In the example illustrated in Figure \ref{systemexample}, the four green railcars on the higher track are grouped together as a single group. Similarly, the four different color-coded railcars on the track below are considered as four separate groups, each consisting of only one railcar.

\begin{figure}[htbp!]
\begin{center}
\includegraphics[scale=0.18]{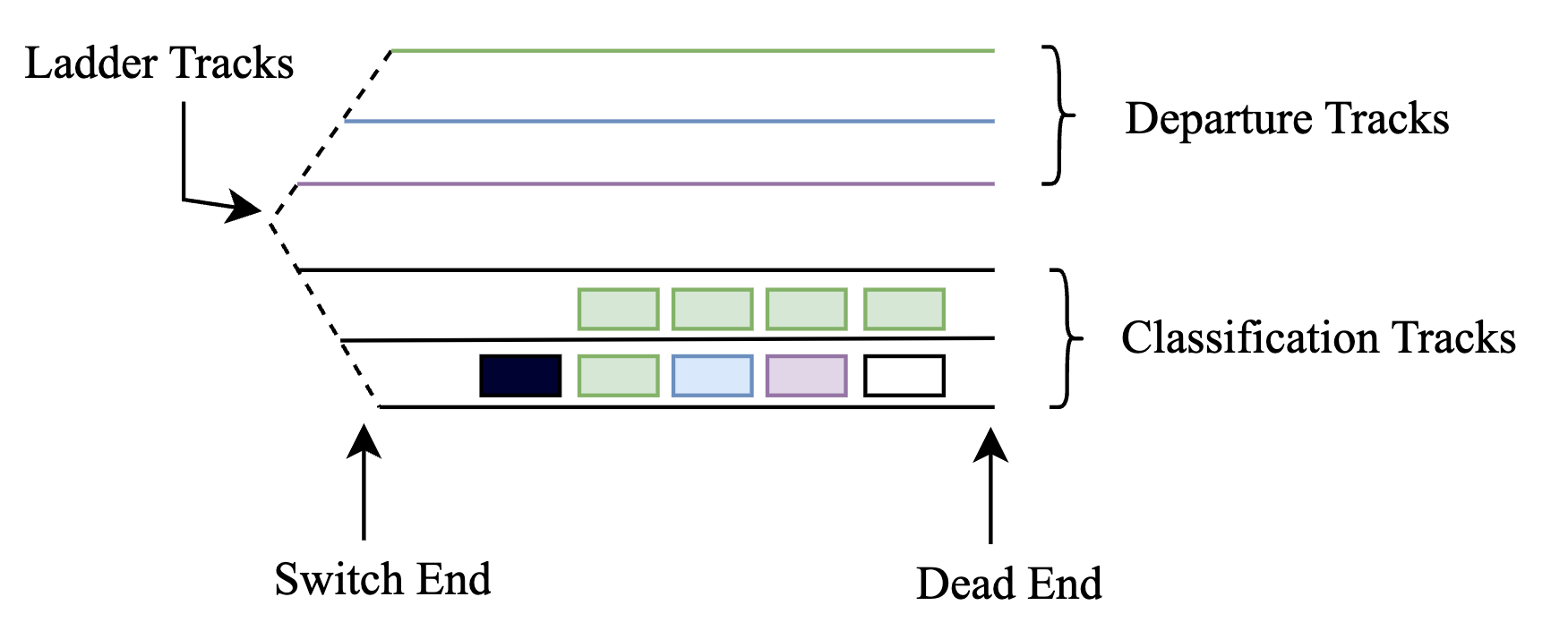}
\caption{Railyard layout example.} \label{systemexample}
\end{center}
\end{figure}

\noindent Building upon the above operational framework, the RSP can be defined as follows.

\begin{definition}[\textbf{RSP}]
    We are given a set of track segments \( K \), partitioned into classification tracks \( K_C \) and departure tracks \( K_D \), and a set of groups \( G \), where each group \( g \in G \) consists of a connected sequence of railcars with length \( L_g \) and the same destination segment \( d(g) \). The set \( G \) is divided into \( G_M \), representing groups assigned to specific departure tracks in \( K_D \), and \( G_N \), representing groups without designated destinations and which may be placed on any classification track in \( K_C \). Each group is mapped to a specific classification track via a function \( f: G \rightarrow K_C \) that designates its initial position on its initial track. The RSP requires determining a sequence of shunting operations that relocates all car groups to their respective destination track segments (in any order) while minimizing the total operational cost. 
\end{definition}

\section{RSP Model Formulations}
\label{sec:model}
As previously discussed, the RSP seeks a shunting sequence that relocates all car groups to their respective destination tracks while minimizing total operational cost. To achieve this goal, we present two exact formulations: an MIP model and a DP model in the following two subsections.

\subsection{MIP model of the RSP}
In this subsection, we describe the MIP model. The required sets and parameters for the MIP model are detailed in Tables \ref{tab:sets} and \ref{tab:parameters}, respectively. Our modeling approach divides time into a set $T$ of discrete time periods, assuming that any move of a subset of cars from one track to another requires one time period. Let $\mathcal{T} = |T|$. The objective of the MIP model is to minimize the cost of relocating all groups to their respective destination tracks over a planning horizon of 
$\mathcal{T}$ periods, indexed by $t$. The choice of an appropriate value of the parameter $\mathcal{T}$ is discussed in Section \ref{sec:Tvalue}.

\begin{table}[htbp]
    \centering
    \footnotesize
    \caption{Set notation summary}
    \label{tab:sets}
    \begin{tabular}{@{}p{1cm} p{12cm}@{}}
        \toprule
        $K$ & Set of track segments, indexed by $i$ and $j$ \\[0.5em]
        $K_C$ & Set of classification track segments \\[0.5em]
        $K_D$ & Set of departure track segments \\[0.5em]
        $G$ & Set of car groups on the tracks, indexed by $g$ \\[0.5em]
        $G_M$ & Set of car groups that have a destination track segment in $K_D$ \\[0.5em]
        $G_N$ & Set of car groups that have a destination track segment in any track in $K_C$ \\[0.5em]
        $T$ & Set of time periods, indexed by $t$ \\[0.5em]
        \bottomrule
    \end{tabular}
\end{table}

\begin{table}[htbp]
    \centering
    \footnotesize
    \caption{Parameter notation summary}
    \label{tab:parameters}
    \begin{tabular}{@{}p{1 cm} p{12cm}@{}}
        \toprule
        $d(g)$ & Destination track segment of car group $g$, $\forall g \in G$ \\[0.5em]
        $L_g$ & Length of car group $g$, $\forall g \in G$ \\[0.5em]
        $\mathcal{L}_i$ & Length of track segment $i$, $\forall i \in K$ \\[0.5em]
        $c_{ij}$ & Cost of moving a car group from track $i$ to $j$, $\forall i,j \in K, j \neq i$ \\[0.5em]
        \bottomrule
    \end{tabular}
\end{table}

We consider a set of track segments $K$ consisting of two disjoint subsets: the classification track segment $K_C$ and the departure track segments $K_D$, where $K_C \cap K_D =\emptyset$ and $K_C \cup K_D=K$. Let $G$ denote the set of groups on the tracks. Among these, $G_M$ is the subset consisting of railcar groups that have an outbound destination track segment in $K_D$, while $G_N$ is the subset of railcar groups without a destination track segment, and thus may be placed on any classification track segment in $K_C$. Since different groups may consist of varying numbers of railcars, we define $L_g$ as the length of group $g$ for all $g \in G$. For each railcar group $g \in G_M$, let $d(g)$ denote its destination track, which is a single track in the set $K_D$. In contrast, for $g \in G_N$, $d(g)$ corresponds to any element of $K_C$. We define the position of group $g$ on its track segment in period $t$ as $P_{gt}$, characterized by its relative distance to the dead end. The group in position 1 is the group closest to the dead end, while the highest position indexed group on a track segment is closest to the switch end. Letting $\mathcal{G} = |G|$, at most $\mathcal{G}$ positions may be occupied on any track segment. To capture how each group's position changes over time, we define $\Delta_{gt}$ as the change in position of group $g$ in period $t$, which is unrestricted in sign. We further introduce two nonnegative variables $\Delta_{gt}^+$ and $\Delta_{gt}^-$ such that $\Delta_{gt}^+=\max\{\Delta_{gt}, 0\}$ and $\Delta_{gt}^-=\max\{-\Delta_{gt},0\}$, so that 
 $\Delta_{gt}=\Delta_{gt}^+-\Delta_{gt}^-$. The parameter \( c_{ij} \) represents the fixed cost for moving a set of car groups from track \( i \) to \( j \).

The model's decision variables are provided in Table \ref{tab:summarydv}. The binary variable $x_{git}$ equals 1 if group  $g$ is located on track segment $i$ at the end of period $t$, and 0 otherwise. We assume that each shunting move requires one time period, and the binary variable $v_{ijt}$ indicates whether the locomotive moves a subset of groups from track segment $i$ to track segment $j$ in period $t$, for all $i,j \in K, i \ne j$ and $t \in T$. The binary variable $z_{gt}$ equals 1 if group $g$ changes track segments in period $t$, and 0 otherwise, while $y_{gijt}=1$ if group $g$ moves from segment $i$ to segment $j$ in period $t$, and 0 otherwise. The nonnegative variable $N_{it}$ represents the number of groups on track segment $i$ at the end of period $t$. To capture the dynamic status of the tracks, we introduce the nonnegative variables  $N_{it}^+$  and  $N_{it}^-$, which denote the number of groups added to and removed from segment $i$ during period $t$, respectively.

\begin{table}[htbp]
    \centering 
    \footnotesize 
    \caption{Decision variable notation summary}
    \label{tab:summarydv}
    \begin{tabular}{ll}
        \toprule
        $x_{git}$: & $\left\{
                    \begin{tabular}{@{}l@{}}
                    1 \text{ if group $g$ is on track segment $i$ at the end of period $t$}, $\forall g \in G, i \in K, t \in T$ \\
                    0 \text{ otherwise}
                    \end{tabular}
                    \right.$ \\ 
        $y_{gijt}$: & $\left\{
                    \begin{tabular}{@{}l@{}}
                    1 \text{  if group $g$ moves from track segment $i$ to track segment $j$ in period $t$}, $\forall g\in G, i, j \in K$, $j \neq i$, $t \in T$\\
                    0 \text{ otherwise}
                    \end{tabular}
                    \right.$ \\ 
        $z_{gt}$: & $\left\{
                    \begin{tabular}{@{}l@{}}
                    1 \text{  if group $g$ changes track segments in period $t$}, $\forall g\in G, t \in T$\\
                    0 \text{ otherwise}
                    \end{tabular}
                    \right.$ \\ 
        $v_{ijt}$: & $\left\{
                    \begin{tabular}{@{}l@{}}
                    1 \text{  if the locomotive moves a subset of groups from track segment $i$ to track segment $j$ in period $t$, } \\ \hspace{0.3cm} $\forall i, j \in K, j \neq i, t \in T$ \\
                    0 \text{ otherwise}
                    \end{tabular}
                    \right.$ \\ 
        $w_{t}$: & $\left\{
                    \begin{tabular}{@{}l@{}}
                    1 \text{  if all groups are in their destination segments at time $t$}, $\forall t \in T$\\
                    0 \text{ otherwise}
                    \end{tabular}
                    \right.$ \\ 
        $u_{t}$: & $\left\{
                    \begin{tabular}{@{}l@{}}
                    1 \text{  if period $t$ is the earliest period at which all groups have reached their destination tracks}, $\forall t \in T$\\
                    0 \text{ otherwise}
                    \end{tabular}
                    \right.$ \\ 
        $N_{it}$: & $
                      \begin{tabular}{@{}l@{}}
                        number of groups on track segment $i$ at the end of period $t$, $\forall i \in K, t \in T$
                      \end{tabular}
                    $ \\  
        $N_{it}^+$: & $
                    \begin{tabular}{@{}l@{}}
                    number of groups added to segment $i$ in period $t$, $\forall i \in K, t \in T$
                    \end{tabular}
                    $ \\ 
        
        $N_{it}^-$: & $
                    \begin{tabular}{@{}l@{}}
                    number of groups removed from segment $i$ in period $t$, $\forall i \in K, t \in T$
                    \end{tabular}
                    $ \\ 

        $\Delta_{gt}$: & $
                    \begin{tabular}{@{}l@{}}
                    change in position of group $g$ in period $t$, $\forall g \in G, t \in T$
                    \end{tabular}
                    $ \\ 
        
        $\Delta_{gt}^+$: & $
                    \begin{tabular}{@{}l@{}}
                    positive change in position of group $g$ in period $t$, $\forall g \in G, t \in T$
                    \end{tabular}
                    $ \\ 
        
        $\Delta_{gt}^-$: & $
                    \begin{tabular}{@{}l@{}}
                    negative change in position of group $g$ in period $t$, $\forall g \in G, t \in T$
                    \end{tabular}
                    $ \\ 
        
        $P_{gt}$: & $
                    \begin{tabular}{@{}l@{}}
                    position of group $g$ on its track segment in period $t$, $\forall g \in G, t \in T$
                    \end{tabular}
                    $ \\
        
        \bottomrule
    \end{tabular}
\end{table}

Given the above definitions, our MIP formulation of the RSP is shown in Table \ref{tab:RSP}.
\begin{table}[htbp]
    \centering 
    \footnotesize 
    \caption{RSP MIP model}
    \label{tab:RSP}
\noindent\rule{\textwidth}{0.4pt} %
\footnotesize
\begin{eqnarray}
 \mbox{Minimize} & \sum\limits_{i\in K}\sum\limits_{\overset{l \in K}{j\ne i}}\sum\limits_{t\in T}c_{ij}v_{ijt}, \label{eq:obj} \\
\mbox{Subject to:} &  \sum\limits_{i\in K} x_{git}=1, &  g \in G, t \in T, \label{eq:two}\\
& \sum\limits_{g \in G} L_g x_{git} \le \mathcal{L}_i, & i \in K, t \in T,\label{eq:3}\\
& \sum\limits_{i \in K}\sum\limits_{\overset{j \in K}{j\ne i}} v_{ijt} \le 1, & t \in T,\label{eq:4}\\
& y_{gijt} \le v_{ijt}, & g \in G, i, j \in K, j \ne i, t \in T,\label{eq:5}\\
& y_{gijt} \le x_{gi,t-1}, & g \in G, i, j \in K, j \ne i, t \in T,\label{eq:6}\\
& N_{it} = N_{i,t-1} - N_{it}^- + N_{it}^+, &   i \in K, t \in T, \label{eq:7}\\
& N_{it}^- = \sum\limits_{g \in G}\sum\limits_{\overset{j \in K}{j\ne i}} y_{gijt}, & i \in K, t \in T,\label{eq:8}\\
& N_{it}^+ = \sum\limits_{g \in G} \sum\limits_{\overset{j \in K}{j\ne i}} y_{gjit}, & i \in K, t \in T,\label{eq:9}\\
& \sum\limits_{i\in K} N_{it}^+ = \sum\limits_{i\in K} N_{it}^- , & t \in T,\label{eq:10}\\
& \Delta_{gt}^+ \le \mathcal{G} \sum\limits_{i \in K}\sum\limits_{\overset{j \in K}{j\ne i}} y_{gijt}, & g \in G, t \in T,\label{eq:11}\\
& \Delta_{gt}^- \le \mathcal{G} \sum\limits_{i \in K}\sum\limits_{\overset{j \in K}{j\ne i}} y_{gijt}, & g \in G, t \in T,\label{eq:12}\\
& \Delta_{gt}^+-\Delta_{gt}^- \ge N_{j,t-1} - N_{i,t-1} + N_{it}^- -M_1 \left(1-y_{gijt}\right), &  g \in G, t\in T, j, i \in K, i \ne j, \label{eq:13}\\
& \Delta_{gt}^+-\Delta_{gt}^- \le N_{j,t-1} - N_{i,t-1} + N_{it}^- + M_1 \left(1-y_{gijt}\right), & g \in G, t\in T, j, i \in K, i \ne j,  \label{eq:14}\\
& P_{gt} = P_{g,t-1} + \Delta_{gt}^+-\Delta_{gt}^- , & g \in G, t \in T,\label{eq:15}\\
& \mathcal{G}y_{gijt} \ge P_{g,t-1} - (N_{i,t-1}-N_{it}^-) - M_2(2 - x_{gi,t-1}-v_{ijt}), & g \in G, i,j \in K, j \ne i, t \in T,\label{eq:16}\\
& \mathcal{G}(1 - y_{gijt}) \ge N_{i,t-1}-N_{it}^- - P_{g,t-1}+1 - M_2(2-x_{gi,t-1}-v_{ijt}), & g \in G, i,j \in K, j \ne i, t \in T,\label{eq:17}\\
& x_{git}\le x_{gi,t-1}+\sum\limits_{\overset{j \in K}{j\ne i}}y_{gjit}, & g \in G, i \in K, t\in T, \label{eq:18}\\
& x_{git}\ge x_{gi,t-1}+2\sum\limits_{\overset{j \in K}{j\ne i}}y_{gjit}-1, & g \in G, i \in K, t\in T,\label{eq:19} \\
& x_{git}\le x_{gi,t-1}-\sum\limits_{\overset{j \in K}{j\ne i}}y_{gijt}+\sum\limits_{\overset{j \in K}{j\ne i}}y_{gjit}, & g \in G, i \in K, t\in T, \label{eq:20}\\
& x_{git}\ge x_{gi,t-1}-\sum\limits_{\overset{j \in K}{j\ne i}}y_{gijt}, & g \in G, i \in K, t\in T,\label{eq:21} \\
& w_{t} \le x_{g,d(g),t}, & t \in T, g \in G_M, \label{eq:22}\\
& w_{t} \le \sum\limits_{i \in K_C} x_{git}, & t \in T, g \in G_N, \label{eq:23} \\
& u_t \le w_t - w_{t-1}, & t \in T, \label{eq:24}\\
& \sum\limits_{t\in T} u_t = 1,\label{eq:25} \\
& \sum\limits_{i\in K}\sum\limits_{\overset{j \in K}{j\ne i}} v_{ij1} \geq \sum\limits_{i\in K}\sum\limits_{\overset{j \in K}{j\ne i}} v_{ij2} \geq \cdots \geq \sum\limits_{i\in K}\sum\limits_{\overset{j \in K}{j\ne i}} v_{ijT}, \label{eq:26}\\
& x_{git}, v_{ijt}, y_{gijt}, w_t, u_t \in \{0, 1\}, & g \in G, i,j \in K, j \ne i, t \in T, \label{eq:27}\\
& N_{it}, N_{it}^+, N_{it}^-, \Delta_{gt}^+,\Delta_{gt}^- , P_{gt} \in \mathbb{Z}^+, &  i \in K, t \in T, g \in G.\label{eq:28}
\end{eqnarray}
\noindent\rule{\textwidth}{0.4pt} %
\end{table}
The objective \eqref{eq:obj} is to minimize the total shunting cost. Constraint set \eqref{eq:two} ensures that every group must occupy exactly one track segment in each period. To obey track length restrictions, we enforce constraint set \eqref{eq:3}. Constraint set \eqref{eq:4} ensures that at most one group can move from one track segment to another in any time period. We also require constraint sets \eqref{eq:5} and \eqref{eq:6} to ensure that a group $g$ cannot move from track segment $i$ to $j$ unless the locomotive moves from $i$ to $j$ in period $t$, and group $g$ was located on track segment $i$ in period $t-1$. Constraint set \eqref{eq:5} also ensures that groups from at most one track segment may change track segments in any period. 

Maintaining track segment group balance and ensuring the same number of groups leave one track and enter another is enforced by constraint sets \eqref{eq:7} -- \eqref{eq:10}. Constraint set \eqref{eq:7} connects the number of groups on track segment $i$ at the end of period $t$ to the number of groups on the same track segment at the end of period $t-1$, by adding the number of groups that arrive on track $i$ and subtracting the number of groups that depart from it. Constraint sets \eqref{eq:8} and \eqref{eq:9} represent, respectively, the number of groups added to track segment $i$ in period $t$, and the number of groups removed from track segment $i$ in period $t$. Constraint set \eqref{eq:10} ensures that in every time period $t$, the total number of groups added across all track segments equals the total number of groups removed from other track segments.

Constraint sets \eqref{eq:11} -- \eqref{eq:15} enforce the proper position changes for a group, given values of the variables $N_{it}^-$ and $N_{jt}^+$. We assume that if $N_{it}^-$ groups are removed from track segment $i$ in period $t$, then these must correspond to the $N_{it}^-$ groups on segment $i$ in period $t$ that are closest to the switch end; these are, therefore, the groups on segment $i$ at time $t$ with the $N_{it}^-$ highest position numbers. Note that $\Delta_{gt} = \Delta_{gt}^+ - \Delta_{gt}^-$ is expressed as the difference of two nonnegative variables, $\Delta_{gt}^+$ and $\Delta_{gt}^-$, which indicate the position change of group $g$ in period $t$. Constraint sets \eqref{eq:11} and \eqref{eq:12} ensure that if group $g$ does not undergo a move in period $t$, its position cannot change. In particular, If $z_{gt} = 0$, these constraints force both $\Delta_{gt}^+$ and $\Delta_{gt}^-$ to 0; if $z_{gt} = 1$, then $\Delta_{gt}^+$ and $\Delta_{gt}^-$ may take positive values up to $\mathcal{G}$. Because the binary variable $z_{gt}=1$ if group $g$ changes track segments in period $t$, and
\begin{eqnarray}
z_{gt} &= \sum\limits_{i \in K}\sum\limits_{j \in K, j \ne i} y_{gijt}, & \forall g \in G, t \in T. \label{eq:29}
\end{eqnarray}
in constraint sets \eqref{eq:11} and \eqref{eq:12}, $z_{gt}$ can be substituted out of the formulation, and replaced with $\sum\limits_{i \in K}\sum\limits_{j \in K, j \ne i} y_{gijt}$. When $z_{gt}=1$, the position change for group $g$ is characterized by Lemma \ref{lem:positionChange}. 

\begin{lemma}\label{lem:positionChange}
Given the number of groups on track segments $i$ and $j$ at the end of period $t-1$, $N_{i,t-1}$ and $N_{j,t-1}$, respectively, if $N_{it}^-$ groups move from segment $i$ to segment $j$ in period $t$, then the resulting position changes are given by:
\begin{eqnarray}
\Delta_{gt} &= N_{j,t-1} - N_{i,t-1} + N_{it}^-,& 
\forall\, g \in G,\; t \in T. \label{eq:30}
\end{eqnarray}
\end{lemma}
\begin{proof}
Please see Appendix \ref{sec:AppendixA}.
\end{proof}

\noindent Constraint sets \eqref{eq:13} and \eqref{eq:14} enforce the correct position change for a group $g$ if it moves from track segment $i$ to $j$ in period $t$, where $M_1$ is a sufficiently large ``big $M$'' value that we can set equal to twice the number of groups (2$\mathcal{G}$) without loss of optimality. Specifically, the right-hand side of constraint set \eqref{eq:13} will be negative if group $g$ does not move from track segment $i$ to $j$ in period $t$, and this constraint is thus redundant. If $y_{gijt}=1$, the term $M_1(1 - y_{gijt})$ vanishes, and constraint sets \eqref{eq:13} and \eqref{eq:14} reduce to enforcing that $\Delta_{gt} = N_{j,t-1} - N_{i,t-1} + N_{it}^-$, thereby ensuring the correct position change for group $g$ in period $t$. Constraint set \eqref{eq:15} updates the resulting position changes in each period.

Constraint sets \eqref{eq:16} and \eqref{eq:17} ensure that the correct groups are selected for a move in any given period ($M_2$ is a sufficiently large ``big $M$'' value that we can set equal to the number of groups, $\mathcal{G}$, without loss of optimality). Recall that constraint sets \eqref{eq:5} and \eqref{eq:6} enforce \(y_{gijt} = 0\) for any track segment pairs \((i,j)\) with \(v_{ijt} = 0\), and for any group/track segment pairs with \(x_{gi,t-1} = 0\). As a result, the binary variable \(y_{gijt}\) can only take the value of 1 if \(v_{ijt} = 1\) and \(x_{gi,t-1} = 1\). In other words, if a move occurs from segment \(i\) to segment \(j\) in period \(t\), then group \(g\) must have been located on segment \(i\) in the previous period $t-1$ and the locomotive must have moved from segment \(i\) to segment \(j\) in period $t$. Both constraint sets are only active for groups $g$ on track segment $i$ at time $t-1$ if there is a move from segment $i$ to $j$ in period $t$; otherwise, the right-hand side of both of these constraint sets will be negative. If there is a move from segment $i$ to segment $j$ in period $t$, then for any group $g$ on segment $i$ at time $t-1$, the right-hand side of the first constraint set will equal $P_{g,t-1} - (N_{i,t-1}-N_{it}^-)$, which is positive for the $N_{it}^-$ highest indexed groups that will move, forcing each corresponding $y_{gijt}$ variable to one.  In this case, the right-hand side of the second constraint set above will equal $N_{i,t-1}-N_{it}^- - P_{g,t-1}+1$, which is positive for the $N_{i,t-1}-N_{it}^-$ lowest indexed groups that will not move, forcing each corresponding $y_{gijt}$ variable to zero. 

Constraint sets \eqref{eq:18}--\eqref{eq:21} update each group’s track assignment as needed in each time period. The main idea is that the binary variable \(x_{git}\) depends on its previous status \(x_{gi,t-1}\) and on the movement decisions \(y_{gjit}\) and \(y_{gijt}\). Specifically, if a group was not on track \(i\) before, it can only arrive if it moves in from some other track; if a group was on track \(i\) before, it can only leave by moving out to another track. In this way, these constraints enforce consistent transitions—remaining on the same track or relocating to a different one—without allowing contradictory assignments. Constraint set \eqref{eq:18} ensures that a group cannot be on track \(i\) at time \(t\) if it was not on track \(i\) in the previous period and does not perform a shunting move onto track \(i\). However, if a group was not on track $i$ in the previous period but does transition onto it at time $t$, constraint set \eqref{eq:19} ensures that the group is on track $i$ in the current period. Similarly, constraint set \eqref{eq:20} ensures that a group cannot be on track \(i\) at time \(t\) if it was on track \(i\) in the previous period and was shunted onto some other track \(j\) in period $t$. However, if a group was on track $i$ in the previous period but does not undergo a shunting move onto another track \(j\), constraint set \eqref{eq:21} ensures that this group remains in place.

Constraint sets \eqref{eq:22}--\eqref{eq:25} terminate the model if all groups are in their destination segments at time $t$. Constraint sets \eqref{eq:22} and \eqref{eq:23} ensure that \(w_{t}=1\) only when every group is in its respective destination at time \(t\). In particular, each group \(g\) in \(G_M\) must occupy its assigned destination track \(d(g)\), and each group \(g\) in \(G_N\) must occupy some track in \(K_C\). Next, constraint set \eqref{eq:24} captures the idea that \(u_{t}\) can only become 1 at the first period when all groups have reached their destinations—i.e., when \(w_{t}\) transitions from 0 to 1. Constraint set \eqref{eq:25} guarantees that exactly one period is identified as the earliest point in time when all groups reach their destination segments. Finally, constraint set \eqref{eq:26} ensures that moves start from time 1 and continue without interruption until all groups have reached their destinations.

Constraint set \eqref{eq:27} enforces the binary requirements for the decision variables \(x_{git}, v_{ijt}, y_{gijt}, w_t,\) and \(u_t\), while constraint set \eqref{eq:28} ensures that the decision variables \(N_{it}, N_{it}^+, N_{it}^-, \Delta_{gt}^+, \Delta_{gt}^-,\) and \(P_{gt}\) are nonnegative integers.

\subsubsection{Determining an appropriate time horizon}
\label{sec:Tvalue}
We next describe how we determine an appropriate value of the time horizon, $\mathcal{T}$, for a problem instance of our RSP model. Determining an appropriate time horizon value is crucial: the horizon must be sufficiently long to ensure finding an optimal solution without leading to unnecessarily high computing time. To achieve this, we employ a fast polynomial-time heuristic that generates candidate feasible solutions detailing group moves in each period, ensuring that all groups ultimately reach their designated tracks. We set the planning horizon $\mathcal{T}$ of the RSP model equal to the total number of shunting moves required by this heuristic solution, as each shunting move requires one time period. Thus, our MIP model determines the minimum cost objective function value within at most $\mathcal{T}$ periods.  For ease of explanation, this heuristic assumes the existence of at least two classification tracks. We also ignore track length constraints, assuming tracks are sufficiently long to accommodate any number of groups (adjustments to the proposed heuristic to account for track length limitations are fairly straightforward assuming a sufficient amount of total track capacity, although we omit the details for the sake of brevity).  

In this context, we define \textit{merge} as the process of combining two car groups that share the same destination into a single group. Additionally, a car group is classified as \textit{switch-end-positioned} if it possesses the highest position index on its track, as \textit{middle-positioned} if its position index is neither the highest nor lowest on its track, and as \textit{dead-end-positioned} if it has an index of one on its track. We index tracks in the set $K$ from $0$ (top) to $n$ (bottom) and let $\mathcal{H}\subseteq K_C$ denote the set of tracks that contain at least one car group in the set $G_M$. With a slight abuse of notation, we denote \(d_s(i)\) as the destination track associated with the switch-end–positioned group on track \(i \in \mathcal{H}\), and let $h_w$ denote the $w^{\rm{th}}$ element of $\mathcal{H}$ for $w=1,\ldots,|\mathcal{H}|$, where $h_1=\{\min i:i\in \mathcal{H}\}$, $h_{|\mathcal{H}|}=\{\max i: i \in \mathcal{H}\}$, and the elements of $\mathcal{H}$ are indexed in increasing order of corresponding track number. Let \(\delta\) represent a difference in track index values, and define

\[
\mathcal{P} = \{(i,j) \in \mathcal{H} : i < j,\; j-i \le \delta,\; d_s(i) = d_s(j)\}.
\]
In other words, \(\mathcal{P}\) is the set of track pairs whose indices differ by no more than \(\delta\) and whose switch-end–positioned groups share the same destination. The heuristic, shown in Algorithm \ref{alg:Theuristic} in Appendix \ref{sec:algorithms}, consists of three main phases: initial preprocessing, adaptively dropping off car groups on destination tracks, and calculation of the time horizon $\mathcal{T}$. 

The initial preprocessing phase aims to form a train that includes all car groups in \( G_M \) while excluding any car groups in \( G_N \) through three key steps. First, on each track, switch-end-positioned groups are merged if they share the same destination and their track differences are within the specified parameter \( \delta \). 
Second, starting with $w=|\mathcal{H}|$, all car groups on track \( h_w \) (except those that are dead-end-positioned and belong to \( G_N \)) are moved to track \( h_{w-1} \).  If $w$ is greater than two, we decrease $w$ by one and repeat, eventually forming a single complete train on track \( h_1 \). Finally, each car group \( g \in G_N \) that is switch-end-positioned or middle-positioned is iteratively shunted to track $h_1+1$. Figures \ref{leftblock} and \ref{middleblock} illustrate the shunting operations for these two types. Each car group \( g \in G_N \) that is switch-end-positioned requires one shunting move. However, each car group \( g \in G_N \) that is middle-positioned requires two shunting moves: the first involves transferring all car groups up to $g$ to the lower track, and the second involves returning the car groups belonging to \(G_M\). This phase ensures that track \(h_1\) ultimately contains only car groups from \( G_M \) and none from \( G_N \).

\begin{figure}[htbp!]
\begin{center}
\includegraphics[scale=0.12]{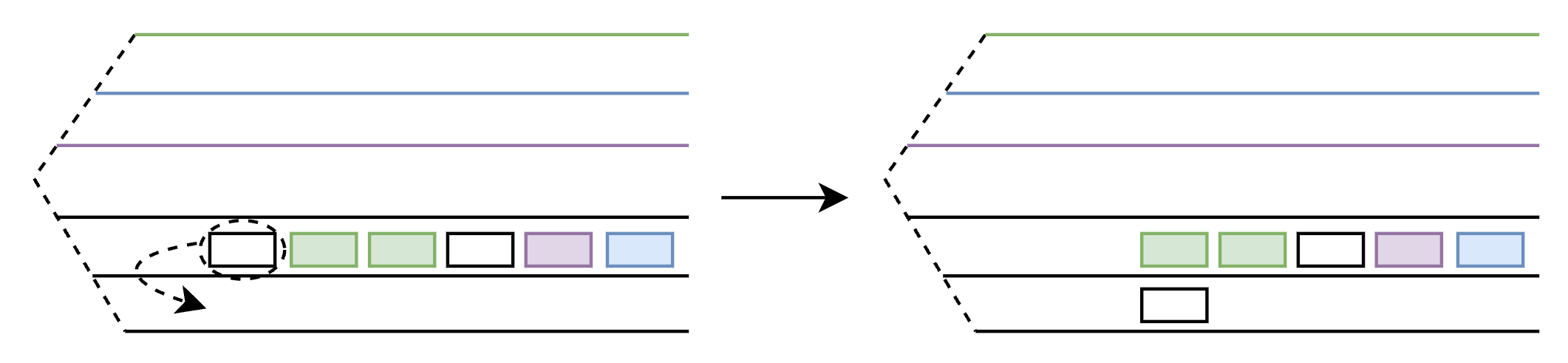}
\caption{Example of shunting operations for switch-end-positioned group $g\in G_N$} \label{leftblock}
\end{center}
\end{figure}

\begin{figure}[htbp!]
\begin{center}
\includegraphics[scale=0.16]{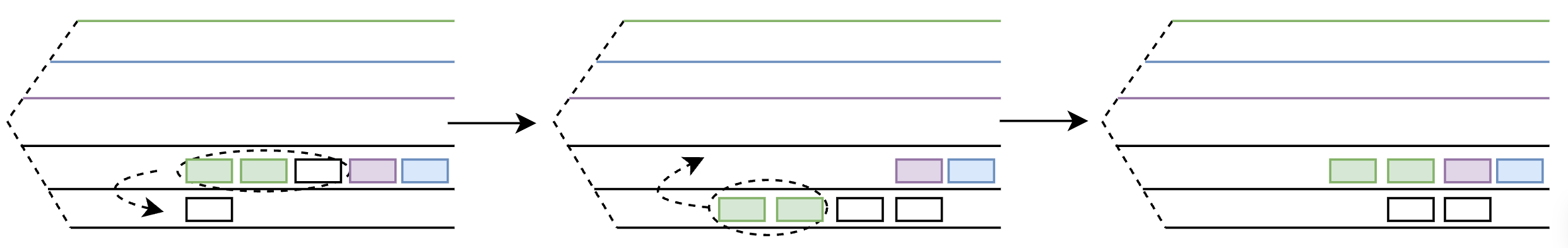}
\caption{Example of shunting operations for middle-positioned group $g\in G_N$} \label{middleblock}
\end{center}
\end{figure}

The heuristic next adaptively moves groups located on track \(h_1\) to their destination tracks by iteratively releasing the dead-end-positioned group onto its designated destination track. Figure \ref{dropoff} illustrates the drop-off operations for a train consisting of four car groups. It is easy to see that the number of shunting moves required for this is equal to the number of car groups \(g\in G_M\).

\begin{figure}[htbp!]
\begin{center}
\includegraphics[scale=0.19]{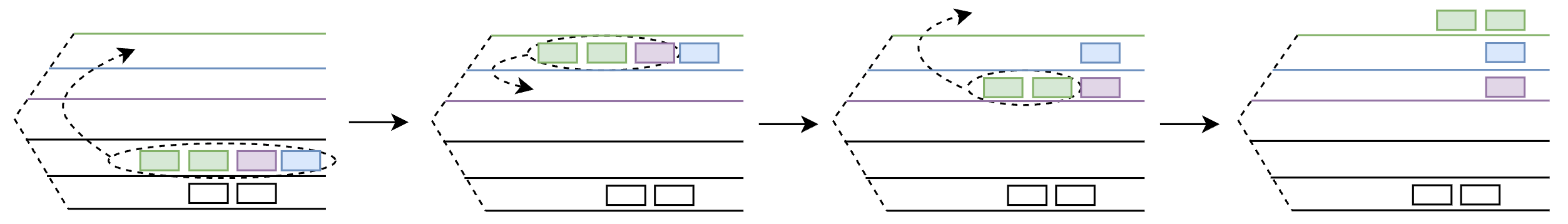}
\caption{Example of shunting operations to drop off dead-end-positioned group $g\in G_M$} \label{dropoff}
\end{center}
\end{figure}

The time horizon length $\mathcal{T}$ is then set equal to the total number of shunting moves made during the initial preprocessing and adaptive car group drop-off processes, since each shunting move requires one time period.

\subsection{DP model of the RSP}
This subsection describes a DP model for the RSP. We retain the use of $K, K_C, K_D, G, g$, and $d(g)$ as in the RSP problem definition. Let \( S \) be the finite set of all possible system states, where each state \( s \in S \) describes, for every track \( k \in K \), an ordered sequence of groups occupying that track. The current location of the locomotive is not included in the state definition. Given a state \( s \in S \), let \( k_g(s) \) be the track on which group \( g \) resides. We specify a final state \(s_F\) in which each car group \(g \in G\) is on its destination track, i.e., 
\begin{align}
k_g(s_F) & = d(g). 
\end{align}
Note that if we create a single final state corresponding to any state in which each group occupies its destination track regardless of group sequence on the track, then we can avoid creating an exponential number of final states.  Our heuristic approximate DP approach discussed in Section \ref{sec:ARG} merges consecutive groups on a common track with the same destination to avoid the possibility of duplicating what are essentially identical states. We begin with an initial state \(s_0\) indicating the current state of the railyard. 

An action consists of selecting one or more car groups and moving them from their current track to another track in a single time period, which corresponds to one shunting move as defined previously. Let \(A_s\) and \(c_s(a)\) denote, respectively, the set of all possible actions in state \(s\) and the cost associated with the action.

For each state \(s\in S\), we let \(V(s)\) be the minimum cost required to reach the final state from \(s\). The transition function \(\mathfrak{t}(a,s)\) determines the state reached from 
\(s\) when action \(a\) is applied. The dynamic programming recursion is defined as follows:

\begin{align}
V(s) &= \min_{a \in A_s} \Bigl\{ c_s(a) + V\bigl(\mathfrak{t}(a, s)\bigr) \Bigr\}. \label{eq:32}
\end{align}

\noindent The solution value of the RSP then equals \(V(s_0)\).

\subsubsection{Directed graph representation of the DP approach}
\label{sec:Graph}
A weighted directed graph \(G(V, E)\) can be constructed associated with our DP model with each vertex corresponding to a particular state \(s \in S\). A directed edge from vertex \(s\) to vertex \(s'\) indicates that a feasible action \(a \in A_{s}\) exists such that \(\mathfrak{t}(a, s) = s'\), with the edge weighted by the associated cost \(c_s(a)\). We construct the graph along a layered horizontal timeline. The initial state vertex \(s_0\) is placed at time layer \(0\). Starting from state \(s_0\), we generate all possible new states reachable using a single action (move), create a vertex corresponding to each of these new states in time layer \(1\), and add a directed edge from \(s_0\) to each generated vertex with weight corresponding to the the value of $c_s(a)$ associated with the transition cost. For each vertex in layer 1, we generate a label containing the transition cost from state $s_0$ and set its predecessor to $s_0$.  The vertex $s_0$ has a label equal to zero and has no predecessor.

After creating a new set of vertices at each layer $\tau\ge 1$, for each vertex in layer \(\tau\), we again enumerate all states reachable with a single action.  If a valid action at a state $s_j$ in layer $\tau$ results in a transition to a state $s_k$ corresponding to an existing vertex in layer $t\le \tau$, then if the label of $s_j$ plus the transition cost is less than the label of $s_k$, we create a directed edge from $s_j$ to $s_k$ and change the label of $s_k$ to equal the label of $s_j$ plus the transition cost (we also change the predecessor of $s_k$ to $s_j$).  Otherwise, if the valid action results in a transition to a new state that has not yet been generated, we create a new vertex corresponding to the new state in layer $\tau+1$, with label value equal to the label of $s_j$ plus the transition cost, and a predecessor of $s_j$. Note that the label value at any point in time corresponds to the length of the shortest path from vertex $s_0$, while the time layer corresponds to the minimum number of moves to reach the given state.

The expansion continues until no new states are produced and all available moves from every existing vertex have been explored. The first time a state corresponding to the final state is generated, a vertex $S_F$ is created. Note that we do not generate successors from vertex $S_F$, but continue expanding other non-final states whose action/transition pairs have not yet been enumerated. A DP network example with one final state is illustrated in Figure \ref{DPnetwork}.

\begin{figure}[htbp!]
\begin{center}
\includegraphics[scale=0.21]{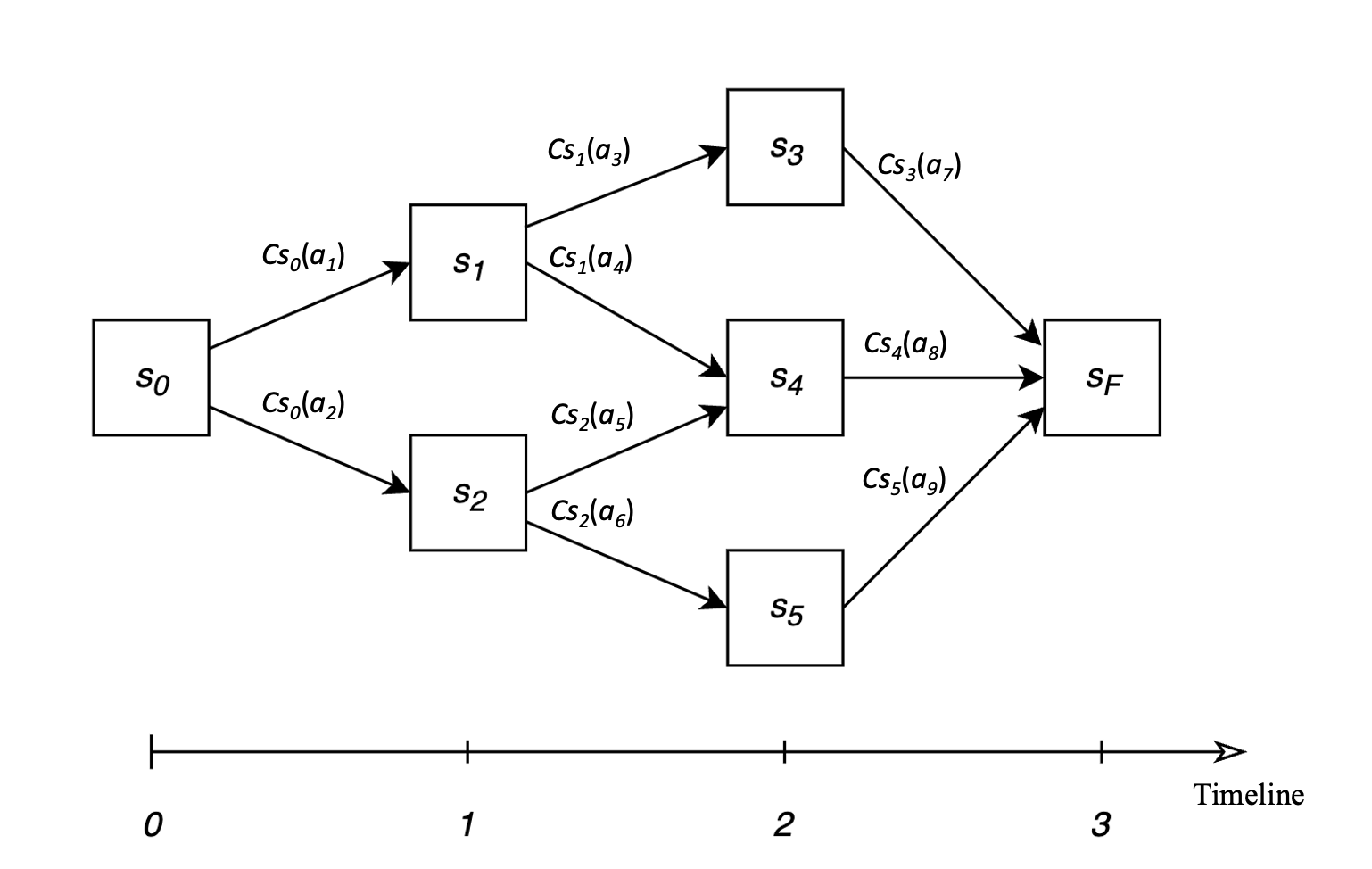}
\caption{DP network example} 
\label{DPnetwork}
\end{center}
\end{figure}

\subsubsection{Total number of states}
\label{sec:States}
In this subsection, letting $n(d)$ denote the number of car groups with destination $d\in K_D$, we characterize the total number of states of the DP network, as detailed in Lemma \ref{lem:DPstates}.

\begin{lemma}\label{lem:DPstates}
Given the number of car groups, $\mathcal{G}$, the number of car groups in $G_N$ (which do not have an outbound destination),  the number of classification tracks \(|K_C|\), and the total number of tracks, \(|K|\), the total number of states $|V|$ in the DP network, representing all possible arrangements of $\mathcal{G}$ car groups across \(|K|\) tracks, allowing for empty tracks, can be expressed as: 

\begin{align}
|V| = \mathcal{G}! \times \binom{\mathcal{G} + |K| - 1}{|K| - 1}-\left\{|G_N|! \times
\binom{|G_N|+|K_C|-1}{|K_C|-1} \times \prod_{d \in K_D}n(d)!-1\right\}.\label{eq:combos}
\end{align}
\end{lemma}

\begin{proof}
Please see Appendix \ref{sec:AppendixA}.
\end{proof}

As shown in Appendix \ref{sec:AppendixA}, the first term on the right-hand side of \eqref{eq:combos} corresponds to the number of ways in which $\mathcal{G}$ different groups can be arranged on $|K|$ tracks. The second term in brackets represents the number of possible arrangements corresponding to the final state, where each group in $G_N$ occupies any track in $K_C$, and each group in $G_M$ has reached its destination track (we can safely eliminate all but one of these states). The total number of states \(|V|\) indicates that using classical algorithms such as  Dijkstra’s algorithm \parencite{dijkstra2022note} to find the shortest path between the initial state \(s_0\) and a final state \(s_F\) is challenging due to the large graph size. Note that reducing the number of car groups and tracks can dramatically reduce the state space. In particular, if the number of car groups with destination $d$ is reduced by \(q(d)\) for each $d\in K_D$, the number of car groups $|G_N|$ reduced by \(p\), and the number of tracks is reduced by $r$, respectively, the corresponding reduction in the total number of states can be quantified as follows, where $q=\sum_{d\in D} q(d)$, $\mathcal{G}_{pq}^-=\mathcal{G}-p-q$, $\mathcal{G}_{N,p}^-=|G_N|-p$,  $\mathcal{K}_r^-=|K|-r$, $\mathcal{K}_{C,r}^-=|K_C|-r$, and $n_q^-(d)=n(d)-q(d)$,

\begin{corollary}[State Space Reduction by Decreasing the Number of Car Groups and Tracks]\label{cor:decreasebygroup_track}
If the number of car groups with destination $d$ is reduced by \(q(d)\) for each $d\in K_D$, (with \(1 \le q(d) \le n(d)\)), the number of car groups $|G_N|$ reduced by \(p\) (with \(1 \le p \le |G_N|\)), and the number of classification tracks is reduced by \(r\) (with \(1 \le r \le |K_C|-1\)), then the new total number of states is
\[
|V^{(p,q,r)}| = \mathcal{G}_{pq}^-! \times \binom{\mathcal{G}_{pq}^-+\mathcal{K}_r^--1}{\mathcal{K}_r^--1} -\left\{ \mathcal{G}_{N,p}^-! \times
\binom{\mathcal{G}_{N,p}^-+\mathcal{K}_{C,r}^--1}{\mathcal{K}_{C,r}^--1} \times\prod_{d\in K_D}n_q^-(d)!-1\right\},
\]
and the ratio of the new total number of states to the original total number of states is given by:
\[
\frac{|V^{(p,q,r)}|}{|V|}=\frac{
\prod\limits_{j=1}^{\mathcal{G}_{pq}^-}(\mathcal{G}_{pq}^-+\mathcal{K}_r^--j)-
\left\{\prod\limits_{i=1}^{\mathcal{G}_{N,p}^-}(\mathcal{G}_{N,p}^-+\mathcal{K}_{C,r}^--i)
\prod\limits_{d\in K_D} n_q^-(d)!-1\right\}}
{\prod\limits_{j=1}^{\mathcal{G}}(\mathcal{G}+|K|-j)-\left\{\prod\limits_{i=1}^{|G_N|}(|G_N|+|K_C|-i)\prod\limits_{d\in K_D} n(d)!-1\right\}},
\]
Consequently, the percentage decrease in the number of states equals
\[
\left(1-\frac{|V^{(p,q,r)}|}{|V|} \right)\times 100\%.  
\]

\end{corollary}

\begin{proof}
Please see Appendix \ref{sec:AppendixA}.
\end{proof}

The above corollary quantifies how the state space shrinks when the number of car groups and number of tracks decrease. For example, if \(\mathcal{G}=|K|=9\), $|G_N|=3,|K_D|=3, |K_C|=6$, and $n(1)=5,n(2)=1,n(3)=6$, reducing the number of car groups \(|G_N|\) by just one (i.e., $p=1$) decreases the total number of states by $94.14\%$. Reducing the number of car groups \(|G_N|\) and tracks each by one (i.e., $p=r=1$) results in a state space reduction of $97.08\%$. Thus, even a small reduction in the number of car groups and/or tracks significantly decreases the DP network size. Motivated by this, the next two sections first establish the \(\mathcal{NP}\)-hardness of the RSP and then propose a heuristic algorithm that uses merging techniques to reduce the number of car groups and eliminate some of the unnecessary empty tracks, thereby enabling an efficient solution via a DP-based heuristic.

\section{Complexity}
\label{sec:complexity}
This section characterizes the complexity of the RSP problem. We begin by defining a special case of the RSP (RSP-sc), and relate it to the problem of zero-shuffle stack sorting (ZSSS). 

\subsection{ZSSS and an RSP-sc}
The stack sorting problem (SSP) begins with an arbitrary sequence of the digits 1 through $n$ on a \emph{source} stack, with $k$ additional stacks available for sorting the numbers. The switch end of a stack is referred to as the top, while the dead end is the bottom. Each number must ultimately enter a separate \emph{sink} stack (only once), and the digits must be sequenced in the sink stack in numerical order from top to bottom (and must, therefore, enter the sink stack in reverse numerical order). A sorting stack can hold any subset of the numbers and can only be accessed from the top, so that items must enter and exit any sorting stack in LIFO order, and we may only move one digit at a time from the top of one stack to the top of some other stack. A shuffle occurs if any one of the digits enters more than one of the $k$ sorting stacks at any point. If, for example, some digit appears above a higher numbered digit on one of the $k$ sorting stacks, then a shuffle is required before the higher numbered digit may enter the sink stack (which must occur before the lower numbered digit can enter). The ZSSS problem asks, for a given initial sequence on the source stack and $k$ sorting stacks, whether a solution exists without any shuffles. Note that the version of the ZSSS problem we have described does not contain a so-called ``midnight constraint,'' where a midnight constraint requires that all items are moved from the source stack to one of the $k$ sorting stacks before any item can be moved to the sink track \parencite{KLstacksorting}. \textcite{EVEN1971} showed that a one-to-one correspondence exists between the ZSSS problem without the midnight constraint, and the problem of finding a $k$-coloring on a circle graph, which was proven to be $\mathcal{NP}$-complete for $k\ge 4$ by \textcite{unger1988k}. Thus, the problem of determining whether a zero-shuffle solution exists for the stack sorting problem with $k\ge 4$ stacks and without a midnight constraint is also $\mathcal{NP}$-complete.

We next define a special case of the RSP that is equivalent to a version of the ZSSS problem in which we may move any consecutive sequence of digits from the top of a stack to the top of another stack (i.e., we do not require moving one digit at a time). 

\begin{definition}[\textbf{RSP-sc}]
Consider an RSP with \(k+2\) classification tracks labeled \(0\) through \(k+1\). We refer to track 0 as the source track and track \(k+1\) as the outbound track, where $n$ car groups with different destinations are located on the source track (and no additional groups are in the system). Suppose the $n$ car groups on the source track are in some sequence \(\sigma_A\) consisting of permutation of the numbers 1 through $n$. Departure tracks are indexed from top to bottom in decreasing order, from \(n'\) to \(1'\), as illustrated in Figure \ref{system}. We define the following cost parameters:
\begin{itemize}
    \item Moving any subset of car groups from track 0 to any track \(1,\ldots,k\) has zero cost; that is, \( c_{ij} = 0 \) for \( i = 0 \), \( j \in \{1,\ldots,k\} \).
    \item Moving any subset of car groups from any track \(1,\ldots,k+1\) to any track \(0,1,\ldots,k\) has cost \(\bar{c} > 0\); that is, $c_{ij}=\bar{c}>0$ for $i\in \{1, \ldots, k+1\}$, $j\in \{0, 1, \ldots, k\}$.
    \item Moving any subset of car groups from any track \(0,\ldots,k\) to any departure track \(1',2', \ldots,n'\) also has cost \(\bar{c} > 0\); that is, $c_{ij}=\bar{c}>0$ for $i\in \{0, \ldots, k\}$, $j\in \{1', 2', \ldots, n'\}$.
    \item Moving any subset of car groups from any track \(1,\ldots,k\) to track \(k+1\) has zero cost; that is, \( c_{ij} = 0 \) for \( i \in \{1,\ldots,k\}  \), \( j=k+1\).
    \item Moving any subset of car groups from track \(k+1\) to the nearest departure track \(1'\) and between any consecutive departure tracks both have cost \(R\), with $R>0$ and $nR < \bar{c}$.
\end{itemize}

Our question is whether the sequence \(\sigma_A\) on track $0$ can be sorted using classification tracks $1,...,k+1$ to form the desired departure track sequence \(\sigma_D\) in decreasing order from the switch end to the dead end on the outbound track, and then moved to the corresponding departure tracks at a total cost equal to $nR$.
\end{definition}

Note that in the RSP-sc, any permutation of the digits from 1 to $n$ may appear on track (stack) 0. We say that a sequence on track 0 is in \emph{correct order} if it is either increasing or decreasing from the switch end to the dead end. If it is in increasing order, we can pull the items out together and move them to one of the tracks $1,\ldots, k$ at zero cost, and then move them one at a time to track $k+1$, also at zero cost; if it is in decreasing order, the items can similarly be moved together to one of the tracks $1,\ldots, k$ at zero cost, and then moved all at once to track $k+1$, also at zero cost. Any sequence that does not follow either of these patterns is considered to be in \emph{incorrect order}. For a sequence with incorrect order, we may need to use multiple sorting tracks to sequence them correctly. This process is illustrated below.

Figure \ref{system} shows three instances of RSP-sc with 5 departure tracks and 6 classification tracks. The only difference is the sequence of car groups. Specifically, we have \(\sigma_A = [1 \ 2 \ 3 \ 4 \ 5]\) in Figure \ref{system}(a), \(\sigma_A = [5 \ 4 \ 3 \ 2 \ 1]\) in Figure \ref{system}(b), and \(\sigma_A = [3 \ 1 \ 5 \ 2 \ 4]\)
in Figure \ref{system}(c). The car groups on track 5 in each figure represent the desired departure sequence \(\sigma_D = [5 \ 4 \ 3 \ 2 \ 1]\). In Figure \ref{system}(a), the initial sequence is in the correct order and we can pull all four groups out together and move them to any one of the tracks 1--4 at zero cost, and then pull them out one-by-one and move them to track 5 to form the decreasing order $[5 \ 4 \ 3 \ 2 \ 1]$ from the switch end to the dead end at zero cost. Finally, we move all five groups together to the departure track $1'$ and drop off group 1 at cost \(R\); then, we move the remaining four groups together to the departure track $2'$ and drop off group 2 at cost \(R\); and so on. In this way, only one classification track among tracks 1 to \(k\) is required for a zero-shuffle solution at a cost of \(5R\). Similarly, in Figure \ref{system}(b) we achieve a zero-shuffle solution at a cost of \(5R\) by moving the whole sequence to one of tracks 1–4 at zero cost, then to track 5 at zero cost, and finally dropping off group 1 at track $1'$, group 2 at track $2'$, etc.

However, in the example in Figure \ref{system}(c), a zero-shuffle solution is not possible using a single track among sorting tracks 1–4. That is, if we sequence these five groups only using a single track among tracks 1–4, we cannot achieve the required decreasing order without shuffling. A zero-shuffle solution using two of the four tracks 1–4 is, however, possible. First, group 3 is moved to track 1, followed by group 1, which is also moved to track 1. Next, group 5 is moved to track 2, and group 2 is then moved to track 2 as well. Then, group 1 is moved from track 1 to track 5, followed by group 2 from track 2 to track 5, and group 3 from track 1 to track 5. Next, group 4 is moved to track 2, followed by a transfer of group 4 from track 2 to track 5, followed by moving group 5 from track 2 to track 5. In this way, we form the required decreasing sequence $[5 \ 4 \ 3 \ 2 \ 1]$ on track 5 without a shuffle. We then drop off group 1 at track \(1'\) at a cost of \(R\), car 2 at track \(2'\) at a cost of \(R\), and continue similarly for the remaining groups, as we did for the previous two instances. Overall, two classification tracks among tracks 1 to \(k\) are required for a zero-shuffle solution at a cost of \(5R\).

\begin{figure}[htbp!]
\begin{center}
\includegraphics[scale=0.19]{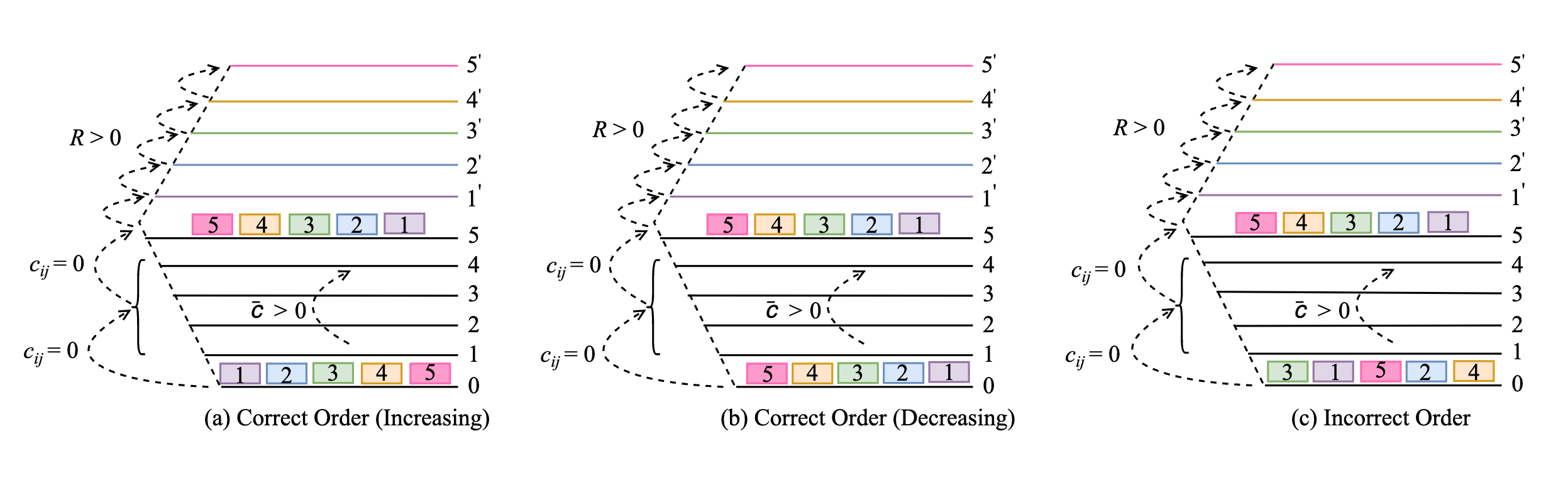}
\caption{Three instances of the RSP-sc with different initial sequences} \label{system}
\end{center}
\end{figure}

In our definition of the RSP-sc, track \(0\) serves as the source stack and track \(k+1\) serves as the sink stack.  The RSP-sc is equivalent to a version of the ZSSS problem that permits a connected set of consecutive car groups to be moved from the top of any stack to the top of any other stack. We demonstrate this equivalence in Appendix \ref{sec:AppNPeq}.  Because the RSP-sc is equivalent to a version of the ZSSS problem in which consecutive item groups can be moved together, we cannot yet claim direct equivalence of the RSP-sc and the ZSSS problem. However, Appendix \ref{sec:AppNPtrans} provides a rather extensive proof that any instance of the ZSSS problem can be transformed in polynomial time to an instance of the RSP-sc, such that a polynomial-time solution for the RSP-sc implies polynomial-time solvability of the ZSSS problem with no midnight constraint, resulting in the following theorem.  
\begin{theorem}\label{th:NP}
The optimization problem RSP with $k\ge 4$ sorting tracks is $\mathcal{NP}$-hard.
\end{theorem}
\noindent In light of Theorem \ref{th:NP} and the corresponding curse of dimensionality implied by the proposed DP formulation in Section \ref{sec:model}, the following section proposes an approximate DP-based heuristic solution method for solving the RSP.

\section{ARG-DP Heuristic}\label{sec:ARG}
Section \hyperref[sec:model]{4} introduced two mathematical models, MIP and DP, for obtaining exact solutions to the RSP. For these models, the problem formulation size and state space, respectively, grow significantly with the number of car groups, the number of tracks, and the length of the planning horizon. Due to the problem’s complexity (as discussed in the previous section) and the scale of practical instances, implementing these models can be computationally prohibitive, and finding an optimal solution may be impossible within acceptable computing time. Because of this, we next propose an ARG-DP heuristic designed to generate solutions efficiently for real-world problem sizes while maintaining an acceptable optimality gap. 

The ARG-DP heuristic framework consists of three main processes: preprocessing, graph construction, and approximate dynamic programming (ADP), as shown in Figure \ref{HF}. Before describing each process, we introduce the following definitions. A \emph{mixed} problem instance is an instance that contains car groups from both \( G_M \) and \( G_N \), while a \emph{non-mixed} problem instance contains only car groups from \( G_M \). For each car group \( g \in G \) and state $s$, we define the distance-to-go, denoted by $\theta_g(s)$, as an estimate of the cost of moving the group from its current track to its target track. This cost is directly proportional to the distance between these tracks. The average distance-to-go for a state, denoted by \( \bar{\theta}(s) \), can be expressed as: \(\bar{\theta}(s) = \frac{1}{\mathcal{G}} \sum_{g \in G} \theta_g(s)\). Let $d_d(i)$ and $d_s(i)$ denote, respectively, the destination track for the \textit{dead-end-positioned} and \textit{switch-end-positioned} groups on track $i$, and define the set of track pairs
\[
\mathcal{Q} = \{(i,j) \in \mathcal{H}: j= i+1,\, d_d(i) = d_s(j),\, \text{and}\; d_s(i) \neq d_d(j) \}.
\]
In other words, \(\mathcal{Q}\) contains pairs of consecutive tracks (i.e., tracks whose indices differ by one) where the destination of the dead-end–positioned group on the first track equals the destination of the switch-end–positioned group on the second track, while the destination of the switch-end–positioned group on the first track differs from that of the dead-end–positioned group on the second track. (Recall that, in contrast, the set of track pairs $\mathcal{P}=\{(i,j)\in \mathcal{H}: i<j, j-i\le \delta\; \text{and}\; d_s(i) = d_s(j)\}$ corresponds to pairs of tracks whose indices differ by no more than \(\delta\) and whose switch-end-positioned groups share the same destination.)

\begin{figure}[htbp!]
\begin{center}
\includegraphics[scale=0.30]{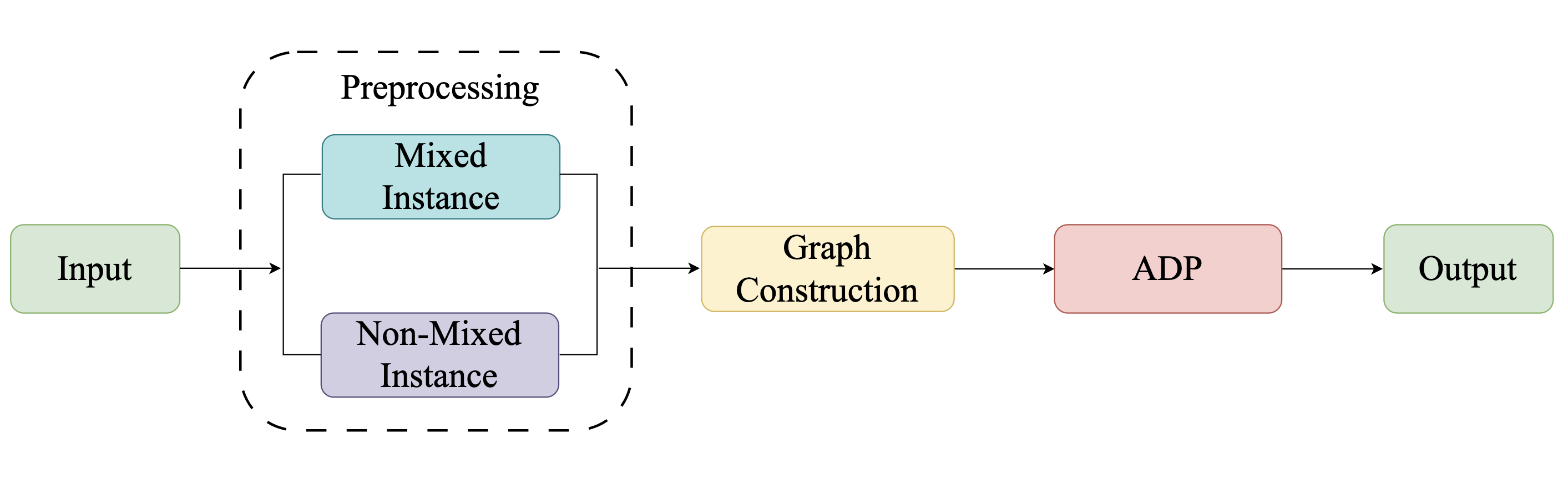}
\caption{ARG-DP heuristic framework.} \label{HF}
\end{center}
\end{figure}

We continue to use the track indexing method introduced in Section \ref{sec:Tvalue}, where tracks are indexed from \(0\) (top) to \(n\) (bottom). We also continue to use \(\mathcal{H}\) to denote the set of tracks containing car groups and sort \(\mathcal{H}\) in ascending order, where \( h_w \) represents the \( w^{\text{th}} \) track in \(\mathcal{H}\), $w=1,\ldots, |\mathcal{H}|$. We define a track as an empty track if it is either completely empty or contains only car groups belonging to \(G_N\).

Recall from  Corollary \ref{cor:decreasebygroup_track} that even a slight reduction in the number of car groups and/or tracks significantly decreases the size of the state space. In the preprocessing stage, the ARG-DP heuristic seeks to merge groups with a common destination (to reduce the number of car groups) and to eliminate a subset of empty tracks (to limit the total number of tracks). In addition, the graph construction phase seeks to reduce the average distance-to-go at each step. To achieve these goals, we first preprocess the initial state and then construct an ADP graph such that only states below a certain threshold value of average cost to go will be considered, thus limiting the resulting problem to a manageable size. 

\subsection{Preprocessing}\label{subsec:pre}
The detailed preprocessing stage for the non-mixed problem instances is presented in Algorithm \ref{alg:nonmixed} in Appendix \ref{sec:algorithms}. Unlike mixed problem instances, non-mixed problem instances contain only car groups from \(G_M\). The preprocessing stage for non-mixed problem instances aims to merge switch-end-positioned groups with common destinations for track pairs in $\mathcal{P}$. If there are no such track pairs in $\mathcal{P}$, we consider merging groups on track pairs in $\mathcal{Q}$ to reduce the total number of groups in the problem instance. Next, in decreasing index order from \(w = |\mathcal{H}|\) to 2, all groups on track \(h_w\) are moved to track \(h_{w-1}\). Notice that these moves help to reduce the size of \(\bar{\theta}(s_{0})\), where \(s_{0}\) is the initial state used as input to the ARG-DP framework. These moves also have the potential to merge the \textit{dead-end-positioned} group on track \(h_w\) with the \textit{switch-end-positioned} group on track \(h_{w-1}\) if they share the same destination. Finally, we form a complete train containing all groups in $G_M$ on track \(h_1\). If \(h_1\) is not equal to $\underline{i}_C=\{\min i:i\in K_C\}$ (i.e., the lowest indexed classification track, which is also the top classification track), the entire train is then moved to track $\underline{i}_C$. This move further decreases the value of \(\bar{\theta}(s_{0})\), as the entire train is moved upward, bringing it closer to the destination tracks.

For mixed problem instances, a car group in \(G_N\) is defined as a \textit{left-blocking} (LB) group if it has no groups to its left and the group to its right is in \(G_M\). A car group in \(G_N\) is defined as a \textit{middle-blocking} (MB) group if the groups both to its left and right are in \(G_M\). Additionally, if multiple MB groups appear on a track, the group with the highest position index is termed the \textit{switch-adjacent middle-blocking} (SA-MB) group. Figure \ref{Blocking} presents an example of these group types. We also define the \textit{neighbor track} of a given track \(i\) (\(0 \leq i \leq n\)) as track \(i-1\). For a given track with LB and SA-MB groups, we use the neighbor track to perform shuffle moves to move these blocking groups to dead-end positions. However, in cases where track \(i+1\) is empty and \(i-1\) is not empty, the neighbor track is instead set to track \(i+1\). Otherwise track $i-1$ remains as the neighbor track. 

\begin{figure}[htbp!]
\begin{center}
\includegraphics[scale=0.16]{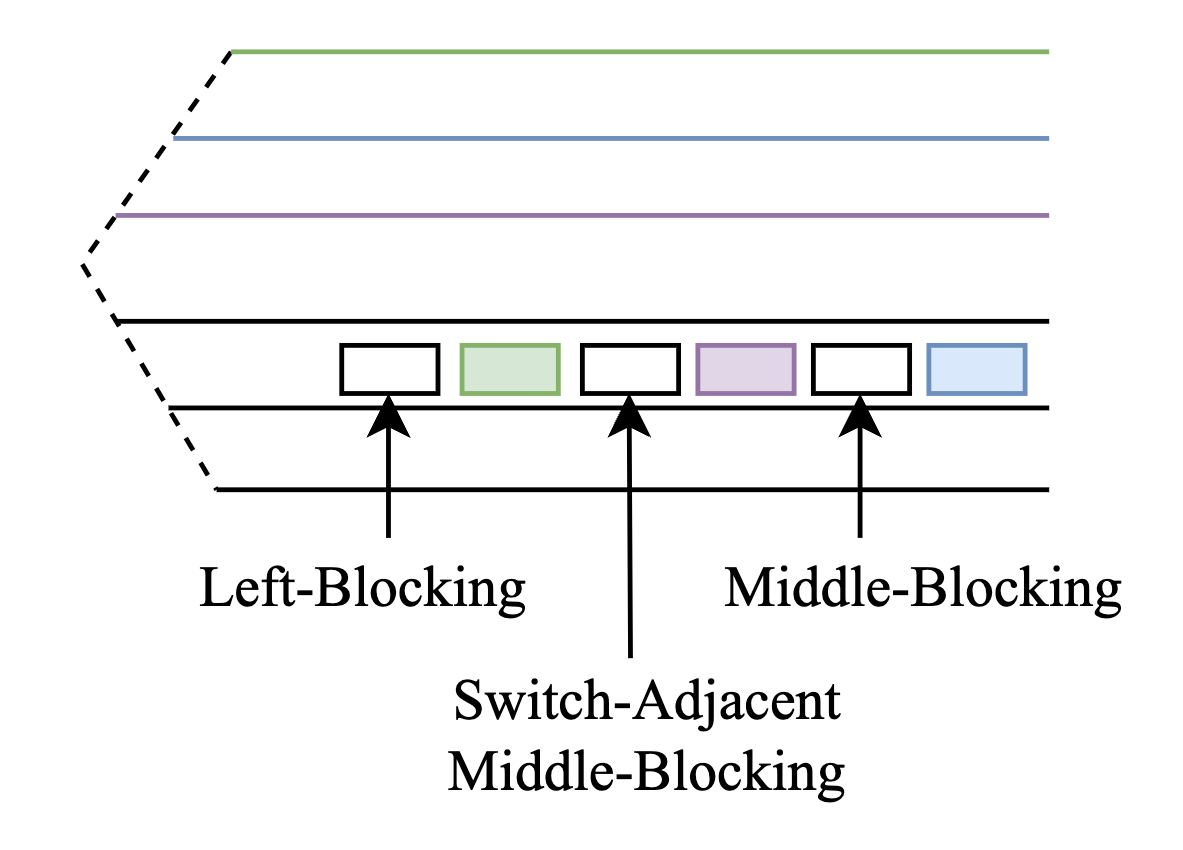}
\caption{Blocking type illustration.} \label{Blocking}
\end{center}
\end{figure}

The preprocessing stage for mixed problem instances aims to move LB and SA-MB groups to the dead end of the same track using shunting movements as outlined in Algorithm \ref{alg:mixed} in Appendix \ref{sec:algorithms}. These groups initially obstruct the movement of car groups in \(G_M\); by moving them to the dead end, they no longer block other groups and can be disregarded, thereby reducing the overall number of car groups in \(G_N\). For an initial state where \(|\mathcal{H}| = 1\) (i.e., only one track contains car groups), we first move the LB group to the neighbor track, and then transfer all groups that are to the left of the SA-MB group—including the SA-MB group itself—to the neighbor track. Finally, we delete all consecutive \textit{dead-end-positioned} \(g\in G_N\). Figure \ref{CBlocking} provides an example of these moves. In the preprocessing stage, we do not consider moving blocking groups beyond the LB and SA-MB groups because the number of steps needed to move such blocks to a dead-end position is greater than one move per blocking group (we therefore handle the resolution of such blocking groups in the ADP algorithm). 

\begin{figure}
\begin{center}
\includegraphics[scale=0.18]{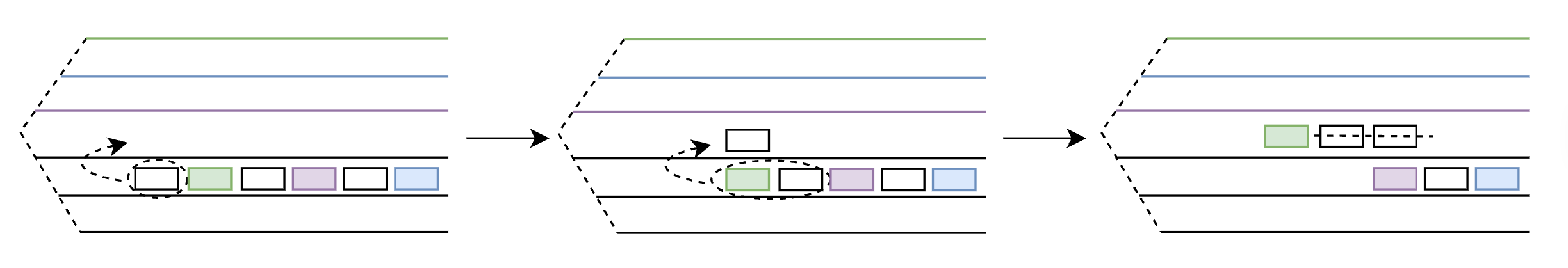}
\caption{Steps taken to clear LB and SA-MB groups when \(|\mathcal{H}| = 1\)} \label{CBlocking}
\end{center}
\end{figure}

For an initial state where \(|\mathcal{H}| > 1\) (i.e., multiple tracks contain car groups), we process tracks in reverse order, from bottom to top. For each $w$ from \(|\mathcal{H}|\) to \(2\), if there is an LB or SA-MB group on track $h_w$, we first move all groups—except any dead-end-positioned groups \(g \in G_N\)—to the neighbor track, leaving only the dead-end-positioned groups in $G_N$ on the current track. This initial movement does not yet remove any LB and/or SA-MB groups. Next, we successively move the LB group and SA-MB group back to track $h_w$, so that these are now dead-end-positioned groups in $G_N$. We apply this processes to tracks \( h_{|\mathcal{H}|}, h_{|\mathcal{H}|-1}, \ldots, h_2\). 

Next, if track \(h_1\) contains an 
LB and/or SA-MB group, we perform a set of additional shunting moves that depend on whether the track immediately below it (i.e., track \(h_1+1\)) is empty.
If track \(h_1+1\) is empty, any LB and/or SA-MB group on track \(h_1\) will be processed by first moving the LB group to track \(h_1+1\), and then moving all groups located to the left of any SA-MB group, including the SA-MB group to track \(h_1+1\). In this case, the LB and SA-MB groups originally on track \(h_1\) become dead-end-positioned groups on track \(h_1+1\).

When track \(h_1+1\) is not empty and we have incurred zero total preprocessing cost, we move all groups on \(h_1+1\), excluding dead-end-positioned groups \(g\in G_N\) to track \(h_1\). After this movement, track \(h_1\) must contain an SA-MB group. We then move all groups located to the left of any SA-MB group, including the SA-MB group to track \(h_1+1\). Finally, we delete all consecutive \textit{dead-end-positioned} groups \(g\in G_N\) on track \(h_1+1\). The detailed shunting movements to clear LB and SA-MB groups in mixed problem instances are presented in Algorithm \ref{alg:mixed} in Appendix \ref{sec:algorithms}.

\subsection{Graph construction and ADP}
The above preprocessing stage for the initial state $s_0$ of mixed and non-mixed problem instances results in an updated initial state $s_0'$. This updated state incurs a preprocessing cost and serves as the input for the graph construction stage. Instead of constructing the full network \(G(V, E)\) described in Section \ref{sec:Graph}, which includes all possible states reachable from \(s_0'\), the graph construction process focuses on building a reduced directed subnetwork \(\hat{G}(\hat{V}, \hat{E})\), where \(\hat{V} \subset V\) and \(\hat{E} \subset E\). Algorithm \ref{alg:ARG} in Appendix \ref{sec:algorithms} presents the graph construction and ADP processes. 

As is described in the algorithm, the subnetwork is constructed iteratively, starting from the initial state \(s_0'\). At each iteration, all new states reachable from the current state are generated. If a state contains consecutive groups sharing the same destination, these groups are merged into a single group. Next, we compare the average distance-to-go of each new state with that of its parent state. A new state is added to the subnetwork if and only if its average distance-to-go is lower than that of its parent. This procedure is applied iteratively to every newly generated state, thereby limiting the total number of states and resulting in a directed partial network that includes the final state among its nodes. 

Based on the constructed subnetwork \(\hat{G}(\hat{V}, \hat{E})\), a shortest path algorithm is applied to compute the minimum cost from the updated initial state \(s_0'\) to a final state. The total cost associated with the RSP solution is the sum of the preprocessing cost and the shortest path cost obtained from the subnetwork, which represents the overall cost of transitioning from the initial state $s_0$ to a final state.

We next discuss the results of a set of computational tests comparing the solution of the RSP using the ARG-DP heuristic with its MIP solution via a commercial solver. 

\section{Computational Results}
\label{sec:results}
We conducted a series of computational experiments to test the performance of the proposed ARG-DP algorithm compared to the MIP model, and provided insights on the RSP.  All tests were performed on a PC equipped with a 3.10 GHz Dual-Core Intel Core i5 processor and 8 GB of RAM. All MIP models of the RSP were solved by Gurobi Optimizer 10.0.1, implemented in Python of version 3.9.13. 

We considered both mixed and non-mixed problem instances to assess whether the presence of groups \( g \in G_N \) affects solution quality. A total of 60 test cases were randomly generated, with 30 classified as mixed problem instances and the remaining 30 as non-mixed.
Table \ref{tab:experimentparameters} presents the parameters for these cases, where the discrete uniform distribution is denoted by $\mathcal{U}(l, u)$  with support over the integers \(l, l + 1, \dots, u\). 

For mixed instances, the total number of car groups consists of car groups \( g \) such that \( g \in G_M \) or \( g \in G_N \). These cases are generated by randomly assigning each group to either \( G_M \) or \( G_N \).  The destination segment of each car group \( g \), where \( g \in G_M \), is randomly assigned a value from \( \mathcal{U}(0, |K_D| - 1) \). The destination segment of each car group \( g \), where \( g \in G_N \), is not unique and can be any value within the range \( [|K_C|, |K| - 1] \).  In contrast, for non-mixed problem instances, all car groups \( g \) belong to \( G_M \), and their destination segments are uniquely assigned a value from \( \mathcal{U}(0, |K_D| - 1) \). The initial positions of these car groups on the track are randomly assigned to one of the classification tracks. 

\begin{table}[htbp]
    \centering
    \footnotesize
    \caption{Parameter settings for computational experiments}
    \label{tab:experimentparameters}  
    \begin{tabular}{@{}p{9cm} p{2cm} p{4cm}@{}}
    
        \toprule
        \textbf{Parameter} & \textbf{Symbol} & \textbf{Value} \\ 
        \midrule
        Number of total tracks & $|K|$ & $\mathcal{U}(4, 10)$ \\
        Number of total departure tracks & $|K_D|$ & $\mathcal{U}(2, \min(|K| - 1, 4))$ \\
        Number of total classification tracks & $|K_C |$ & $|K| - |K_D|$ \\
        Number of total car groups & $|G|$ & $\mathcal{U}(2, 9)$ \\
        Number of total car groups for $G_N$ for non-mixed instances  & $|G_N|$ & $0$ \\
        Number of total car groups for $G_N$ for mixed instances  & $|G_N|$ & $\mathcal{U}(1, 3)$ \\
        Number of total car groups for $G_M$ & $|G_M|$ & $|G| - |G_N|$\\
        Length of car group $g$, $\forall g \in G$ & $L_g$ & 1 \\
        Length of track segment $i$, $\forall i \in K$ & $\mathcal{L}_i$ & $|G|$ \\
        Cost of moving car groups from track $i$ to $j$, $\forall i,j \in K, j \neq i$& $c_{ij}$ & $|i-j|$ \\
        Destination segment of car group $g$, $\forall g \in G$ (non-mixed instances) or $ g \in G_M \text{ (mixed instances)}$ & $d(g)$&$\mathcal{U}(0, (|K_D| - 1))$\\
        Destination segment of mixed instances' car group $g$, $\forall g \in G_N$& $d(g)$& \( [|K_C|, |K| - 1] \)\\
        Track pair difference & $\delta$ & 2\\
        \bottomrule
    \end{tabular}
\end{table}

To evaluate the performance of the proposed ARG-DP algorithm compared to the MIP model, we compute the optimality gap as the difference between the ARG-DP cost and the optimal cost, divided by the optimal cost. The \emph{state reduction} is calculated as 
\( \frac{|V| - |\hat{V}|}{|V|}\times 100\% \), which represents the percentage decrease in the number of states when employing the ARG-DP algorithm compared to the exact DP model. The time ratio equals the MIP running time divided by the corresponding ARG-DP running time. Note that the heuristic shown in Algorithm \ref{alg:Theuristic} takes under 0.01 seconds to determine the time horizon for each case; therefore, the running time of the MIP model does not include the computational time for this heuristic.

Table \ref{tab:states_reduction_results} summarizes the average values of \(|V|\), \(|\hat{V}|\), and state reduction percentage for mixed and non-mixed problem instances specifically. The value of \(|V|\) is calculated directly using Lemma \ref{lem:DPstates}, based on the number of car groups and tracks provided in $s_0$, and \(|\hat{V}|\) is the ADP network size. On average, the full state network contains over 5.8 million states, highlighting the computational complexity of the original RSP instances. The ARG-DP algorithm significantly reduces the network size, yielding an average of 2,290 states across all 60 instances, corresponding to an overall state reduction of 97.92\%. Specifically, the reduction reaches 99.34\% for mixed instances on average and 96.50\% for non-mixed instances on average, demonstrating the effectiveness of the ARG-DP algorithm in reducing the state space.

\begin{table}[htbp]
    \centering 
    \footnotesize 
    \caption{State reduction results of the ARG-DP heuristic}
    \label{tab:states_reduction_results}
    \begin{tabular}{lcccc}
        \toprule
        Test Instances & \(|V|\) & \(|\hat{V}|\) & State Reduction (\%)\\
        \midrule
        Mixed & 7,419,262.73 & 3,160.97 & 99.34 \\
        Non-Mixed & 4,207,118.47 & 1,419.03 & 96.50 \\
        \cmidrule(lr){1-4}
        \textbf{Overall Avg.} & 5,813,190.60 & 2,290.00 & 97.92 \\ 
        \bottomrule
    \end{tabular}
\end{table}

Table \ref{tab:results} summarizes the average values for the objective function, running time (in seconds), time ratio, percentage of problems solved optimally (\% Opt.), and optimality gap (\%). Additionally, the table presents the maximum running times for both the MIP model and the ARG-DP algorithm, along with the maximum optimality gap specifically for the ARG-DP algorithm. Note that for the ARG-DP algorithm, the reported objective function value represents the sum of the preprocessing cost and the shortest path cost.  For full details on all 60 test instances, please refer to Appendix \ref{sec:Computational_Details}. As Table \ref{tab:results} shows, the maximum running time for mixed problem instances is 4774.88 seconds using the MIP model, compared to only 203.26 seconds with the ARG-DP algorithm. Similarly, for non-mixed problem instances, the MIP model takes up to 1130.67 seconds, whereas the ARG-DP algorithm requires a maximum of 472.21 seconds. Figure \ref{Performance} illustrates the results for mixed and non-mixed problem instances specifically. As the figure shows, the ARG-DP algorithm provides solutions in extremely short running times compared with the MIP solutions via Gurobi. Specifically, solving the MIP problems takes nearly 500 seconds for the mixed instances and 200 seconds for the non-mixed instances on average, whereas the ARG-DP algorithm solves both in under a minute. Table \ref{tab:results} further highlights the advantages of the ARG-DP algorithm in terms of solution quality. Among all 60 instances tested, the ARG-DP heuristic found an optimal solution in 60\% of the cases, and provided an average optimality gap of 5.05\% for non-mixed problem instances and 8.24\% for mixed instances. Note that the objective function values take relatively small integer values, with an average value of about seven track moves (i.e., moves of groups between tracks). Thus, an integer solution that deviates from optimality by one unit results in about a 14.3\% optimality gap. We next present the results of application of our solution approaches using the layout of an actual railyard in North Portugal. 

\begin{table}[htbp]
    \centering 
    \footnotesize 
    \caption{MIP and ARG-DP heuristic results}
    \label{tab:results}
   \begin{tabular}{l*{10}{c}}
        \toprule
        & \multicolumn{4}{c}{MIP} & \multicolumn{6}{c}{ARG-DP} \\
        \cmidrule(lr){2-5} \cmidrule(lr){6-10}
         &  &      &       &  Max  & &   &    &  Opt.   & Max    & Max \\
         &  & Time & Time  & Time & & Time & \% & Gap & Gap & Time   \\ 
        Instances & Obj. & (sec)  & Ratio &  (sec) & Obj. & (sec) & Opt. & (\%) & (\%) & (sec)  \\
        \midrule
        Mixed & 6.97 & 485.83 &160.12 &4,774.88 & 7.53 & 16.70 & 56.67  & 8.24 & 33.33 & 203.26  \\
        Non-Mixed & 7.07 & 194.12 & 549.98 &1,130.67& 7.47 & 31.13 & 63.33  & 5.05 &28.57 &472.21 \\
        \midrule
        \textbf{Average} & 7.02 & 339.98 & 355.05 & & 7.50  &23.92 & 60.00  & 6.65  \\ 
        \bottomrule
    \end{tabular}
\end{table}

\begin{figure}[htbp!]
\begin{center}
\includegraphics[scale=0.6]{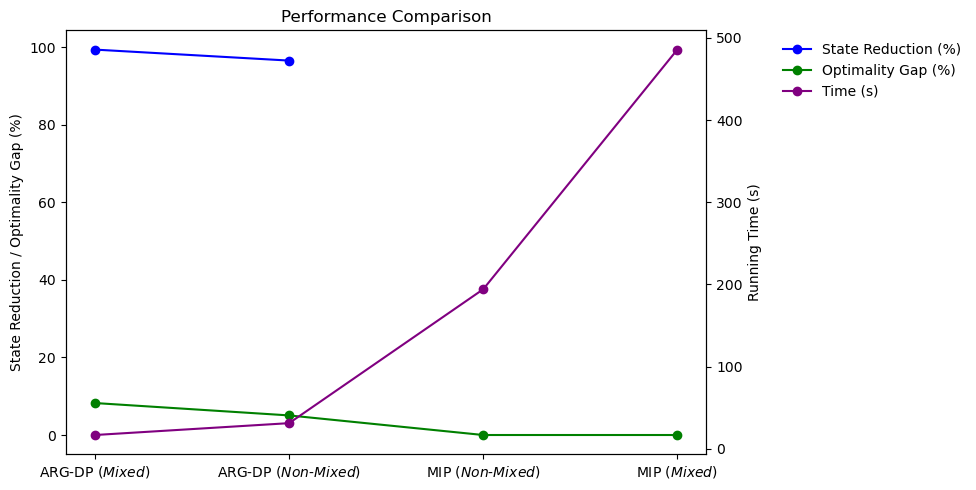}
\caption{Average performance comparison for the 60 cases} 
\label{Performance}
\end{center}
\end{figure}

\subsection{Mini case study}
\label{sec:casestudy}
To gauge the performance of our solution methods on a practical problem, we conducted a mini case study using the structure of the Gaia flat-shunted yard in North Portugal \autocite{Gaiayard}. The Gaia yard includes distinct areas for arrival, departure, shunting, and temporary storage of freight cars. In particular, the shunting zone includes four electrified tracks and ten non-electrified tracks; a locomotive is used for shunting movements on these non-electrified tracks. A single shunting crew with one locomotive performs all yard movements, making efficient shunting tasks critical. In addition, four tracks serve as arrival and departure tracks.

We generated a total of 10 test cases, with five classified as mixed problem instances (Gaia-mixed) and the remaining five as non-mixed problem instances (Gaia-non-mixed). Unlike the previous 60 test cases, this study fixed the number of classification and departure tracks based on the Gaia yard structure depicted by \textcite{Gaiayard}, containing ten classification tracks and four departure tracks. As shown in Figure \ref{Gaia}, Tracks 0 to 3 are designated as departure tracks, while Tracks 4 to 13 serve as classification tracks. All other parameters were generated as shown in Table \ref{tab:experimentparameters}.

\begin{figure}[htbp!]
\centering
\includegraphics[scale=0.22]{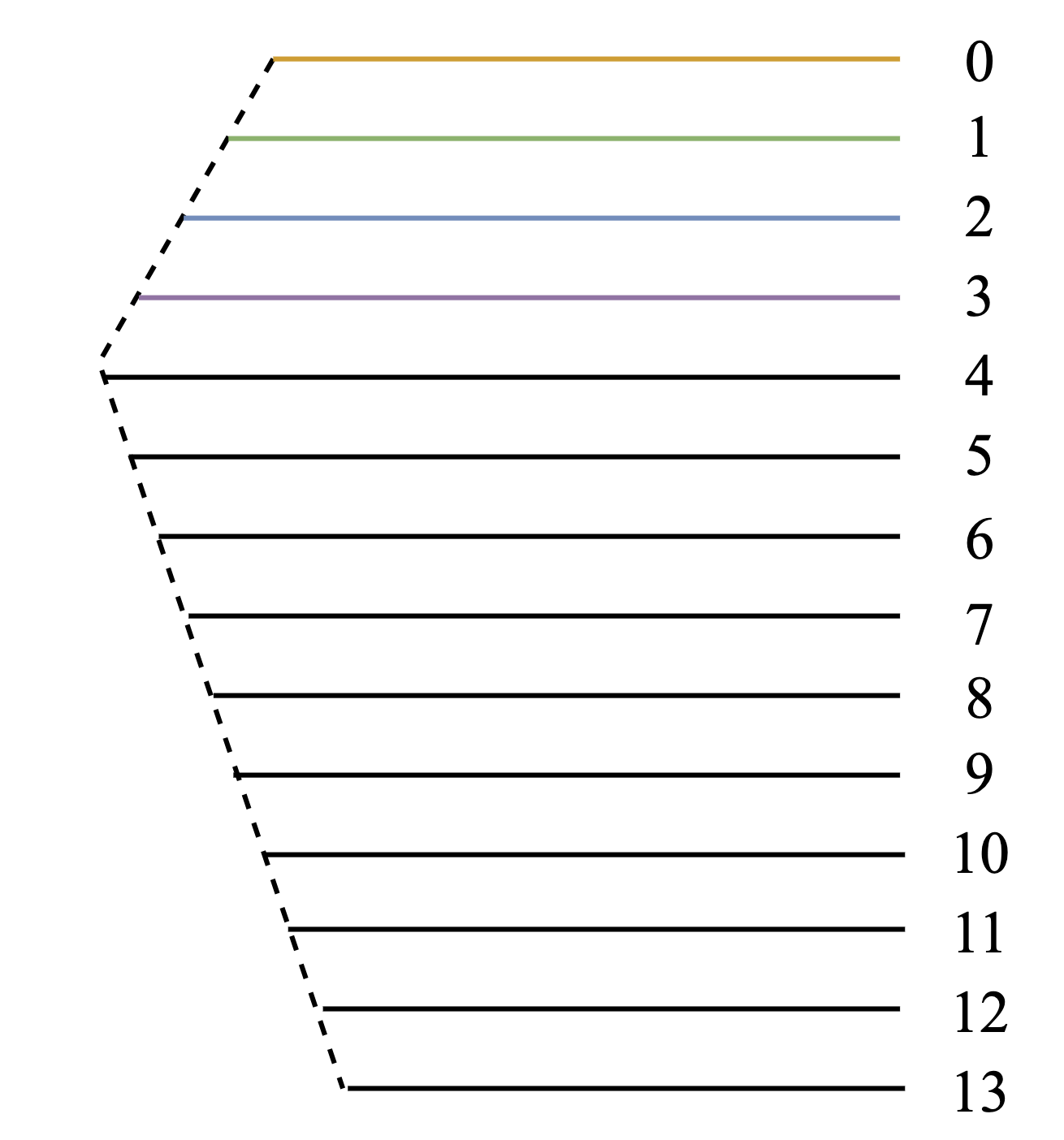}
\caption{Structure of the Gaia flat-shunted yard} 
\label{Gaia}
\end{figure}

The results are presented in Tables \ref{tab:states_reduction_results_gaia} and \ref{tab:case_results}. Table \ref{tab:states_reduction_results_gaia} reports the original state space size \(|V|\), computed based on the number of car groups and tracks in each test case, along with the reduced subnetwork size \(|\hat{V}|\) obtained from the ADP network and the corresponding state reduction percentage. Across all 10 tested cases, the ARG-DP algorithm achieves a significant state space reduction, averaging 99.57\%, with reductions of 99.26\% for Gaia-mixed cases and 99.88\% for Gaia-non-mixed cases. These results demonstrate the effectiveness of the ARG-DP algorithm in reducing the state space.

\begin{table}[htbp]
    \centering 
    \footnotesize 
    \caption{State reduction results for the Gaia flat-shunted yard}
    \label{tab:states_reduction_results_gaia}
   \begin{tabular}{lcccc}
        \toprule
        Test Cases & \(|V|\) & \(|\hat{V}|\) & State Reduction (\%)\\
        \midrule
        Gaia-Mixed & 4,330,277.00 & 2,057.00 & 99.26 \\
        Gaia-Non-Mixed & 37,745,616,364.00 & 8,251.60 & 99.88 \\
        \cmidrule(lr){1-4}
        \textbf{Overall Avg.} & 18,874,973,320.50 & 5,154.30 & 99.57 \\ 
        \bottomrule
    \end{tabular}
\end{table}

Table \ref{tab:case_results} presents the average corresponding objective values, running times, time ratios, percentage of problems solved optimally (\% Opt.), and optimality gaps. In addition, the maximum running times for both the MIP model and the ARG-DP algorithm, along with the maximum optimality gap specifically for the ARG-DP algorithm are presented. For full details on all 10 test cases, please refer to Appendix \ref{sec:Computational_Details}. For the Gaia-mixed cases, the ARG-DP algorithm completes in 31.34 seconds on average, while the MIP solver takes 1342.07 seconds on average. The ARG-DP algorithm obtains solutions 280.56 times faster than the MIP solver on average. In the Gaia-non-mixed cases, the average running times for both the MIP and ARG-DP algorithm are longer than those for the Gaia-mixed cases. This is because the average number of states for the Gaia-non-mixed test cases is greater than that for the Gaia-mixed cases. As shown in Table \ref{tab:states_reduction_results_gaia}, the average reduced subnetwork size for Gaia-non-mixed cases is 8251.60, which is over four times larger than that for Gaia-mixed cases. Across all 10 tested instances, the ARG-DP algorithm achieves an average running time of 227.27 seconds, compared to 8948.37 seconds for the MIP solver, demonstrating significantly faster performance. At the same time, the ARG-DP heuristic solved half of the problems to optimality and resulted in an average optimality gap of 6.90\%. These results allow us to conclude that the ARG-DP algorithm becomes very promising for quickly solving real-world problem instances.

\begin{table}[htbp]
    \centering 
    \footnotesize 
    \caption{Results for the Gaia flat-shunted yard}
    \label{tab:case_results}
    \begin{tabular}{l*{10}{c}}
        \toprule
        & \multicolumn{4}{c}{MIP} & \multicolumn{6}{c}{ARG-DP} \\
        \cmidrule(lr){2-5} \cmidrule(lr){6-11}
        &  &  &   &Max & & &  &Opt. &Max &Max  \\ 
        &  & Time & Time  &Time & & Time & \% & Gap &Gap &Time  \\ 
        Instances & Obj. & (sec)  & Ratio & (sec) & Obj. & (sec) & Opt.  & (\%) & (\%)& (sec)\\
        \midrule
        Gaia-Mixed     & 10.00  & 1,342.07  & 280.56  & 6,477.66  & 10.60  & 31.34 & 60.00 & 4.87   & 16.67  & 128.92  \\
        Gaia-Non-Mixed & 12.00  & 16,554.68 & 177.45  & 50,932.42  & 13.20  & 423.19 & 40.00 & 8.93   & 25.00  & 1,208.01 \\
        \midrule
        \textbf{Average} & 11.00  & 8,948.38  & 229.01 &   & 11.90  & 227.27 & 50.00 & 6.90   &   &   \\
        \bottomrule
    \end{tabular}
\end{table}

\section{Conclusion and Future Research}
\label{sec:conclusion}
This paper presents a decision model for railcar movement in yards with one-sided track access, which operate according to a LIFO property (similar to a stack), addressing problems that arise in freight railyards at manufacturing facilities. We propose both MIP and DP models for minimizing shunting costs. Rather than handling individual railcars, our models manage car groups—sets of consecutive railcars sharing the same destination—as unified entities. Both models allow for simultaneous movement of multiple car groups, enhancing operational efficiency. A weighted directed graph is constructed for solving the DP model, with each vertex representing a particular state. Each edge indicates a valid transition between states and is weighted by its associated shunting cost. The total number of states depends on the number of car groups and tracks. We characterize how the DP network size decreases when the number of car groups and tracks decrease. The complexity of the RSP is formally established by demonstrating the equivalence of a special case to the ZSSS problem and an equivalent $k$-coloring problem. Given this complexity and the need to reduce the number of network nodes, we introduce the ARG-DP heuristic. This heuristic considers state characteristics (including as mixed versus non-mixed cases) during a set of preprocessing steps, applies dynamic grouping strategies with merging techniques to reduce the number of car groups during graph construction, and further utilizes a reduced distance-to-go function to construct the graph. Based on the constructed graph, we find the shortest path (shunting cost) from the initial state to the final state. Computational experiments illustrated the efficiency of the ARG-DP heuristic. Across 60 simulated test instances, the ARG-DP heuristic produced solutions on average \(3.55 \times 10^2\) times faster than the proposed MIP model, while maintaining an average optimality gap of 6.65\%.  Furthermore, a mini case study was conducted using the operational structure of the Gaia flat-shunted yard in Northern Portugal, which includes four tracks serving as both arrival and departure tracks, and ten classification tracks used for shunting operations. In this yard, a single shunting crew with one locomotive handles all yard movements.
In this case study, five mixed and five non-mixed problem instances were generated. The ARG-DP heuristic obtained results \(2.29 \times 10^2\) times faster than the MIP method, maintaining an average optimality gap of 6.90\%. For all 70 test instances, the average state reduction exceeds 95\%, primarily due to the use of merging techniques and the reduced distance-to-go function, both of which effectively limit the graph size. 

Future work may extend the RSP to settings where departure tracks are not predetermined, allowing all tracks to serve as both classification and departure tracks, with an objective of consolidating railcar groups of the same destination onto any single track. Another direction is to optimize shunting operations on two-sided access tracks with FIFO queue structures, where groups are added at one end and removed from the other. Furthermore, the MIP and DP model of the RSP can be extended to include multiple locomotives, requiring coordinated shunting operations within the yard. 

\printbibliography

\newpage
\noindent {\Large \textbf{Online Supplementary Material}}
\begin{appendices}
\section{Proofs of Lemmas and Corollaries} 
\label{sec:AppendixA}
\begin{proof}[\unskip\nopunct] \noindent \textbf{Proof of Lemma \ref{lem:positionChange}.} Consider a group $g$ that is among the $N_{kt}^-$ groups moving from track segment $k$ to segment $l$ in period $t$. Because exactly $N_{kt}^-$ groups leave segment $k$ in period $t$, these must be the $N_{kt}^-$ groups occupying the highest position indices on segment $k$. In particular,
\[
P_{g,t-1} > N_{k,t-1} - N_{kt}^-.
\]
The new position of group $g$ on segment $l$ at time $t$ is therefore
\[
P_{g,t} = N_{l,t-1} \;+\; P_{g,t-1} \;-\; 
\bigl(N_{k,t-1} \;-\; N_{kt}^-\bigr).
\]
Hence,
\[
P_{g,t} - P_{g,t-1} = N_{l,t-1} \;-\; \bigl(
    N_{k,t-1} \;-\; N_{kt}^-\bigr),
\]
which simplifies to 
\[
\Delta_{gt} = N_{l,t-1} \;-\; N_{k,t-1} \;+\; N_{kt}^-,
\]
thereby completing the proof.
\end{proof}

\begin{proof}[\unskip\nopunct] \noindent \textbf{Proof of Lemma \ref{lem:DPstates}}. We count the number of states in two independent steps. First, since the $\mathcal{G}$ car groups are distinct, they can be arranged in a sequence in $\mathcal{G}!$ different ways. Second, once the car groups are ordered, we partition the sequence into $|K|$ segments (each representing a track) by inserting $|K|-1$ dividers. Let $x_i\ge0$ denote the number of car groups assigned to track $i$, for $i=1,2,\dots,|K|$, so that
\begin{align}
x_1+x_2+\cdots+x_{|K|}=\mathcal{G}.
\end{align}
The number of solutions to this equation (i.e., the number of ways to distribute $\mathcal{G}$ car groups into $|K|$ tracks allowing for empty tracks) is given by 

\[
\left(\!\!\binom{|K|}{\mathcal{G}}\!\!\right),
\]
where, according to Identity 143 in \textcite{benjamin2022proofs}, 
\[
\left(\!\!\binom{|K|}{\mathcal{G}}\!\!\right)=\binom{\mathcal{G}+|K|-1}{\mathcal{G}}.
\]
By the symmetry of binomial coefficients, 
\[
\binom{\mathcal{G}+|K|-1}{\mathcal{G}}=\binom{\mathcal{G}+|K|-1}{|K|-1}.
\]
Thus, the number of ways to partition the ordered sequence is 
\[
\binom{\mathcal{G}+|K|-1}{|K|-1}.
\]
Since the ordering of the car groups and the partitioning into tracks are independent, the total number of possible states is given by
\[
|V_0|=\mathcal{G}!\times\binom{\mathcal{G}+|K|-1}{|K|-1}.
\]
Note that the above total number of possible states \(|V_0|\) contains duplicate final states. We next provide a method 
to eliminate these duplications. For any state equivalent to the final state, car groups $g \in G_N$ may be arranged in any order on the $|K_C|$ classification tracks. The number of ways to arrange $|G_N|$ car groups on $|K_C|$ classification tracks is 
\[
|V_1|=|G_N|!\times \binom{|G_N|+|K_C|-1}{|K_C|-1}.
\]
For car groups \( g \in G_M \), each group must be assigned to its corresponding departure track based on its destination. Considering the relative order of car groups in each destination track, the number of distinct arrangements is  
\[
|V_2| = \prod_{d \in K_D} n(d)!.
\]
Since $|V_1|$ and $|V_2|$ are independent, the total number of duplicate final states is $|V_1|\times|V_2|-1$. Therefore, by eliminating duplicate final states, the total number of states equals
\[
|V|=|V_0| - \{|V_1|\times|V_2| - 1\}
=
\mathcal{G}! \times \binom{\mathcal{G} + |K| - 1}{|K| - 1}-\left\{|G_N|! \times
\binom{|G_N|+|K_C|-1}{|K_C|-1} \times \prod_{d \in K_D}n(d)!-1\right\}.
\]  
\end{proof}

\begin{proof}[\unskip\nopunct] \noindent \textbf{Proof of Corollary \ref{cor:decreasebygroup_track}.} Clearly, from the proof of Lemma \ref{lem:DPstates} above, when we have $\mathcal{G}_{pq}^-$ groups and $\mathcal{K}_r^-$ tracks, we have 
\begin{align*}
|V^{(p,q,r)}|=\mathcal{G}_{pq}^-! \times \binom{\mathcal{G}_{pq}^-+\mathcal{K}_r^--1}{\mathcal{K}_r^--1} -\left\{ \mathcal{G}_{N,p}^-! \times
\binom{\mathcal{G}_{N,p}^-+\mathcal{K}_{C,r}^--1}{\mathcal{K}_{C,r}^--1} \times\prod_{d\in K_D}n_q^-(d)!-1\right\}
\end{align*}
Noting that 
\[
\mathcal{G}!\times\binom{\mathcal{G}+|K|-1}{|K|-1}=(\mathcal{G}+|K|-1)(\mathcal{G}+|K|-2)\times \cdots \times |K|=\prod\nolimits_{j=1}^{\mathcal{G}}(\mathcal{G}+|K|-j),
\]
we have
\[
\frac{|V^{(p,q,r)}|}{|V|}=\frac{
\prod\limits_{j=1}^{\mathcal{G}_{pq}^-}(\mathcal{G}_{pq}^-+\mathcal{K}_r^--j)-
\left\{\prod\limits_{i=1}^{\mathcal{G}_{N,p}^-}(\mathcal{G}_{N,p}^-+\mathcal{K}_{C,r}^--i)
\prod\limits_{d\in K_D} n_q^-(d)!-1\right\}}
{\prod\limits_{j=1}^{\mathcal{G}}(\mathcal{G}+|K|-j)-\left\{\prod\limits_{i=1}^{|G_N|}(|G_N|+|K_C|-i)\prod\limits_{d\in K_D} n(d)!-1\right\}}.
\]
\end{proof}

\section{Algorithms}\label{sec:algorithms}
\begin{algorithm}[H]
\caption{Heuristic algorithm to find $\mathcal{T}$}\label{alg:Theuristic}
\begin{algorithmic}[1]
\State Initialize the number of shunting moves \(s \gets 0\).
\State \textbf{Initial preprocessing:}

\medskip
\State \textit{Check for pairs in \(\mathcal{P}\):}
\For{\(w \in \text{range}(|\mathcal{H}|, 1, -1)\)} 
    \State Let \(j \gets h_w.\)
    \State Let \(i = \{\max k: (k,j) \in \mathcal{P}\}.\) 
    \If{such \(i\) exists}
        \State \multiline{Execute a shunting move to merge the switch-end-positioned group on track \(j\) with that on track \(i\).}
        \State \(s \gets s + 1.\)
    \EndIf
\EndFor

\medskip
\State \textit{Move groups sequentially upward:}
\For{\(w \in \text{range}(|\mathcal{H}|, 1, -1)\)} 
    \State Identify all groups on track \(h_w\), excluding dead-end-positioned groups \(g\in G_N.\) 
    \State Execute a shunting move for the identified groups from \(h_w\) to track \(h_{w-1}\).
    \State \(s \gets s + 1.\)
\EndFor

\medskip
\State \textit{Clear any \textit{switch-end-positioned} group \(g\in G_N\) from Track \(h_1\):}
\If{a switch-end-positioned group \(g \in G_N\) exists on track \(h_1\)}
    \State Execute a shunting move to move \(g\) to track \(h_1+1\).
    \State \(s \gets s + 1.\)
\EndIf

\medskip
\State \textit{Clear any \textit\textit{middle-positioned} groups \(g\in G_N\) from Track \(h_1\):}
\For{each middle-positioned group \( g\in G_N\)}
    \State Execute two shunting moves to clear it from track \(h_1\).
    \State \(s \gets s + 2.\)
\EndFor

\medskip
\State \textbf{Adaptively drop off car groups on destination tracks:}
\For{each dead-end-positioned group \(g\)}
    \State Execute a shunting move to drop off group \(g\) at track \(d(g)\).
    \State \(s \gets s + 1.\)
\EndFor

\medskip
\State \textbf{Time horizon calculation:}
\State \(\mathcal{T} \gets s\).
\end{algorithmic}
\end{algorithm}

\begin{algorithm}[H]
\caption{Preprocessing for non-mixed problem instances}\label{alg:nonmixed}
\begin{algorithmic}[1]
\State Initialize preprocessing cost \(P_c \gets 0\).
\medskip
\State \textbf{Check for \(\mathcal{P}\) pairs:}
\For{\(w \in \text{range}(|\mathcal{H}|, 1, -1)\)} 
    \State Let \(j \gets h_w.\)
    \State Let \(i = \{\max k: (k,j) \in \mathcal{P}\}.\) 
    \If{such \(i\) exists}
        \State \multiline{Execute a shunting move to merge the switch-end-positioned group on track \(j\) with that on track \(i\).}
        \State Update preprocessing cost: \(P_c \gets P_c + c_{ij}.\)
    \EndIf
\EndFor

\medskip
\State \textbf{Check for \(\mathcal{Q}\) pairs:}
\If{\(P_c = 0\)}
    \For{\(w = 1\) to \(|\mathcal{H}|-1\)}
    \State Let \(i \gets h_w.\)
    \If{\((i, i+1) \in \mathcal{Q}\)}
        \State \multiline{Execute a shunting move to move all groups on track \(i\) to track \(i+1\).}
        \State Update preprocessing cost: \(P_c \gets P_c + c_{i,i+1}.\)
    \EndIf
\EndFor
\EndIf

\medskip
\State \textbf{Move groups upward:}
\For{\(w \in \text{range}(|\mathcal{H}|, 1, -1)\)} 
    \State Identify all groups on track \(h_w\).
    \State Move these groups from track \(h_w\) to track \(h_{w-1}\).
    \State Update preprocessing cost: \(P_c \gets P_c + c_{h_w,h_{w-1}}\).
\EndFor
\If{track \(h_1 \ne \underline{i}_C\)}
    \State Move all groups on track \(h_1\) to track $\underline{i}_C$.
    \State \(P_c \gets P_c + c_{h_1, \underline{i}_C}\).
\EndIf

\medskip
\State \textbf{Delete empty tracks:}\\
Delete empty tracks with an index greater than \(\underline{i}_C+1\).

\end{algorithmic}
\end{algorithm}

\begin{algorithm}[H]
\caption{Preprocessing for mixed problem instances}\label{alg:mixed}
\begin{algorithmic}[1]
\State Initialize preprocessing cost \(P_c \gets 0\).
\State Let \(h_w^n\) denote the neighbor track of \(h_w\) for \(w = 1, 2, \ldots, |\mathcal{H}|\).
\medskip
\If{\(|\mathcal{H}| = 1\)}
    \State Move the LB group on track \(h_1\) to the neighbor track \(h_1^n\).
    \State \(P_c \gets P_c + c_{h_1,h_1^n}\).
    \State \multiline{Move all groups located to the left of the SA-MB group, including the SA-MB group itself, to the neighbor track.}
    \State \(P_c \gets P_c + c_{h_1,h_1^n}\).
    \State Delete all consecutive dead-end-positioned groups in \(G_N\) on the neighbor track \(h_1^n\).
\Else
    \For{\(w \in \text{range}(|\mathcal{H}|, 1, -1)\)}
        \If{an LB or SA-MB group exists on track \(h_w\)}
            \State Identify all groups on \(h_w\), excluding dead-end-positioned groups \(g\in G_N\).
            \State Move the identified groups from track \(h_w\) to track \(h_w^n\).
            \State \(P_c \gets P_c + c_{h_w,h_w^n}\).
            \State Move the LB group from \(h_w^n\) to track \(h_w\).
            \State \(P_c \gets P_c + c_{h_w^n,h_w}\).
            \State \multiline{Move all groups located to the left of the SA-MB group, including the SA-MB group itself from track $h_w^n$ to track \(h_w\).}
            \State \(P_c \gets P_c + c_{h_w^n,h_w}\).
        \EndIf
    \EndFor
\EndIf
    \medskip 
    \If{\(w = 1\)}
        \If{an LB or SA-MB group exists on track \(h_1\)}
            \If{track \(h_1+1\) is empty}
                \State Move the LB group from track \(h_1\) to track \(h_1+1\).
                \State \(P_c \gets P_c + c_{h_1,h_1+1}\).
                \State \multiline{Move all groups located to the left of the SA-MB group, including the SA-MB group itself from track \(h_1\) to track \(h_1+1\).}
                \State \(P_c \gets P_c + c_{h_1,h_1+1}\).
            \NoThenElseIf{track \(h_1+1\) is not empty and \(P_c = 0\) \textbf{then}}
                \State Identify all groups on \(h_1+1\), excluding dead-end-positioned groups \(g\in G_N\).
                \State Move the identified groups to track \(h_1\). 
                \State \(P_c \gets P_c + c_{h_1+1,h_1}\).
                \If{an SA-MB group exists on track \(h_1\)}
                    \State \multiline{Move all groups located to the left of the SA-MB group, including the SA-MB group itself from the track \(h_1\) to track \(h_1+1\).}
                    \State \(P_c \gets P_c + c_{h_1,h_1+1}\).
                    \EndIf
            \EndIf
        \EndIf
    \EndIf
    \State Delete all consecutive dead-end-positioned groups in \(G_N\).
\end{algorithmic}
\end{algorithm}

\begin{algorithm}[H]
\caption{Graph construction and ADP}\label{alg:ARG}
\begin{algorithmic}[1]

\medskip
\State \textbf{Graph construction:}
\For{each vertex in the current graph}
    \State Generate all possible subsequent vertices.
    \If{a generated vertex is new (i.e., it has not previously appeared in the graph)}
        \If{two adjacent groups share the same destination}
            \State Merge these groups into a single group and retain only the merged state.
        \EndIf
        \If{the new vertex has a lower average distance-to-go than its parent state}
            \State Retain the new vertex.
        \Else
            \State Discard the new vertex.
        \EndIf
    \EndIf
    \State \multiline{Add a directed edge from the current vertex to each retained new vertex with the corresponding weight.}
\EndFor
\State Repeat the above steps until no new vertices can be generated.

\medskip
\State \textbf{ADP:}
\State Initialize a distance array \texttt{dist} with infinity, except that the source vertex is set to 0.
\For{each edge \((p,q)\) in the graph}
    \State \texttt{dist}[$q$] \(\gets \min\big(\texttt{dist}[q],\, \texttt{dist}[p] + \texttt{weight}(p,q)\big)\).
\EndFor
\State For vertex \(q\) corresponding to the final state, record the \texttt{dist}[$q$] as the shortest path cost.
\end{algorithmic}
\end{algorithm}

\section{Complexity Results}
\label{sec:AppendixNP}

\subsection{Equivalence of RSP-sc and Generalization of ZSSS Problem}\label{sec:AppNPeq}
Suppose that in the RSP-sc the given sequence \(\sigma_A\) on track $0$ can be sorted using classification tracks $1,...,k$ to form the desired departure sequence \(\sigma_D\) in decreasing order from the switch end to the dead end on track $k+1$, and then moved to the corresponding departure tracks at a cost equal to $nR$. According to the cost parameters defined, the sorting part involves 0 cost, implying that no groups moved from any track $1,...,k+1$ to any track $0,...,k$. Groups located on classification tracks $1,...,k$ are moved to the outbound track $k+1$ at zero cost. Finally, all groups located on the outbound track $k+1$ are then moved to the departure tracks $1,2,...,n$ successively and dropped off in the sequence $1,2,...,n$, incurring a total cost of $nR$. Thus, we conclude that if the given sequence on track $0$ can be sorted using the $k+1$ classification tracks to form the departure sequence at a cost of $nR$, then no groups were moved from tracks $1,...,k+1$ to any track $0,...,k$, and a zero-shuffle solution exists for the variant of stack sorting where multiple consecutive items can be moved together and no midnight constraint exists.
 
Next, suppose a zero-shuffle solution exists for stack sorting when multiple consecutive items can be moved together (and without the midnight constraint). This implies that the groups in sequence $\sigma_A$ can be sorted using the $k$ classification (sorting) tracks and merged onto the classification (sink) track $k+1$ in decreasing order. The sort and merge processes incur zero cost. To reach the final departure track, all groups located on track $k+1$ will be moved to their corresponding departure tracks sequentially at a cost of $nR$, for a total cost of $nR$. Therefore, if a zero-shuffle solution exists for stack sorting when multiple consecutive items can be moved together and with no midnight constraint, then the given sequence on track $0$ can be sorted and deposited on track $k+1$ to form the departure sequence at a cost of $nR$.

\subsection{Equivalence between the ZSSS problem and the RSP-sc}\label{sec:AppNPtrans}
\textcite{KLstacksorting} showed that approximating the minimum number of shuffles in the SSP problem with $k\ge 4$ stacks is $\mathcal{NP}$-hard. We next demonstrate a polynomial transformation from an instance of their problem to an instance of the RSP-sc.

In their statement of the problem, a random initial sequence $[a_1 \ a_2 \ \cdots \ a_n]$ of $n$ groups must be sorted using $k$ sorting stacks to obtain the departing sequence $[1 \ 2 \ \cdots \ n]$. A so-called \emph{conflict graph} is created that represents an equivalent instance of the $k$-coloring problem on a circle graph (the $k$-coloring problem seeks the minimum number of node colors such that no two nodes sharing an edge have the same color). A conflict exists between two items (and a corresponding edge between two nodes in the conflict graph) if they cannot be placed on the same stack without requiring a shuffle, while a zero-sort shuffling solution using $k$ stacks exists if and only if a $k$-coloring is possible on the corresponding conflict graph. For the ZSSS problem, a conflict exists in the initial sequence $[a_1 \ a_2 \ a_3 \ \cdots \ a_n]$ between $a_i$ and $a_{i+v}$ for $v \ge 1$ if and only if (i) $a_{i+v} > a_i$ and (ii) not all numbers less than $a_i$ appear somewhere in the sequence $[a_1\ a_2\ \cdots\ a_{i+v-1}]$. 

In contrast, for the RSP-sc problem, Appendix \ref{sec:AppNPconflict} shows that a conflict exists in the initial sequence $[a_1 \ a_2 \ a_3 \ \dots \ a_n]$ between two groups $a_i$ and $a_{i+v}$ (where $1 \leq i < i+v \leq n$) if and only if each of the four conditions shown in Table \ref{tab:conflict-conditions} holds simultaneously (Appendix \ref{sec:AppNPconflict} provides a detailed characterization of the conflict graph properties for the RSP). 

To clarify these conditions, we first define a \emph{consecutive decreasing subsequence} (CDS) of digits as any subsequence
$[a_i\ a_{i+1}\ \cdots\ a_{i+v}]$
such that for every \(j=1,\dots,v\), 
\[a_{i+j}=a_{i+j-1}-1.\]
For example, the subsequence $[8\ 7\ 6\ 5]$ is a CDS of length four, while any single digit is considered to be a CDS of length one.  As another example, the subsequence $[4\ 3\ 12\ 11\ 10]$ contains two consecutive CDSs, one of length two followed by one of length three. Observe that when multiple consecutive items can be moved together from the top of one stack to the top of another, a sequence of CDSs in which the first number in each sequence is increasing can be placed on a sorting stack without requiring a shuffle. In other words, for the sequence $[4\ 3\ 12\ 11\ 10]$, the subsequence $[4\ 3]$ can go to the sink using one move (after the sink contains the subsequence $[2\ 1]$), and the subsequence $[12\ 11\ 10]$ can later go to the sink using one move (after $[9\ 8\ \cdots\ 2\ 1]$ appears on the sink stack). 

\begin{table}[h!]
\centering
\caption{Conditions defining a conflict between $a_i$ and $a_{i+v}$ in the RSP-sc}
\label{tab:conflict-conditions}
\begin{tabular}{|c|p{13cm}|}
\hline
\textbf{Condition} & \textbf{Description}\\
\hline
(i) & $a_i < a_{i+v} - 1, 1< v\le n-1-i$. \\
\hline
(ii) & Not all numbers smaller than $a_i$ appear before $a_{i+v}$. \\
\hline
(iii) & The subsequence $[a_i\ a_{i+1}\ \cdots, a_{i+v}]$ in $\sigma_A$ does not consist of a sequence of CDSs (possibly of length 1) such that the first digit in each CDS is increasing. \\
\hline
(iv) & The subsequence $[a_i\ a_{i}+1\ \cdots\ a_{i+v}-1\ a_{i+v}]$ of $\sigma_A$ is not a sequence of CDSs (possibly of length 1) such that the first digit in each CDS is increasing. \\
\hline
\end{tabular}

\end{table}

Conditions (i) and (ii) in Table \ref{tab:conflict-conditions} represent a slight modification of the conflict conditions required for the ZSSS problem by \textcite{KLstacksorting}, in which Condition (i) becomes $a_i<a_{i+v}$ and Condition (ii) is the same. (Condition (i) must change because, as is shown in Appendix \ref{sec:AppNPconflict}, a number $i$ cannot conflict with $i-1$ or $i+1$ in the RSP-sc, which creates the need for Steps 1 - 4 in the following transformation procedure.)

Given an instance of the ZSSS problem with initial sequence $\sigma_A$, we transform the instance to an equivalent instance of the RSP-sc using the following steps, after initializing $i=1$, $\lambda=0$, and $V=\emptyset$.  

\begin{enumerate}[Step 1.]
\item Find all pairs of consecutive digits $(j,j+1)$ such that there is a conflict in $\sigma_A$ for the ZSSS problem between the two numbers (implying that $j+1$ appears after $j$ and not all digits smaller than $j$ appear before $j+1$ in $\sigma_A$). Denote the lowest valued pair as \( (j, j+1)\) and the highest number in the sequence as $m$.

\item Add 1 to each number from \(j+1\) to $m$. Add \(j+1\) to the set $V$.

\item Repeat Steps 1 and 2 until no pair of consecutive digits exists with a conflict. 


\item Insert each item in $V$ immediately after the last digit in the sequence less than or equal to that number.

\item Find the smallest index \( j > 1 \) such that there is a conflict between \( a_i \) and \( a_{i+j} \) in the ZSSS problem instance. If there is a conflict between \( a_i \) and \( a_{i+j} \), set \( \lambda \leftarrow \lambda+1 \); if no such conflict exists, let \( i \leftarrow i+1 \) and repeat. 

\item Increase each number greater than or equal to $\lambda$ by 1, and place a dummy group immediately before group \( a_{i+j} \).

\item If $i=n-2$ continue to Step 9.  Otherwise let $i\leftarrow i+1$ and return to Step 6.  

\item Number the dummy groups sequentially from \( 1 \) to \( \lambda \) from left to right in the final sequence.

\item If a dummy group number $k$ was inserted to create a conflict between $a_i$ and $a_{i+j}$, and $[a_i\ \cdots\ a_{i+j}]$ forms a sequence of CDSs whose first number is increasing, increase all numbers in the series that are greater than $k$ by 1 and insert the dummy group number $k+1$ immediately to the right of dummy group $k$. 
\end{enumerate}

As is shown in Lemma \ref{lem:dum}, the artificial groups in the set $V$ added in Step 5 do not create any new conflicts, but permit creating conflicts between successively numbered groups by inserting dummy groups in Step 7.  Steps 5 through 8 create conflicts between groups where a conflict exists in $\sigma_A$ for the ZSSS problem, but where one does not (yet) exist for the RSP-sc.  Step 9 in the above procedure accounts for the unlikely case in which a dummy group $k$ was inserted immediately after a real group number $k+1$, in which case Condition (iii) in Table \ref{tab:conflict-conditions} could still be violated after Steps 1--8.

We next provide an example to illustrate the above transformation. Consider the initial sequence 
$$\sigma_A=[2\ 4\ 3\ 6\ 1\ 8\ 7\ 5]$$
in the SSP described by \textcite{KLstacksorting}, whose conflict graph is shown in Figure \ref{KLsequence} (their example assumes $k=2$ stacks; because a 2-coloring is not possible for this graph, at least one shuffle is required using two stacks and a ZSSS is not possible on two stacks). As we transform the group numbers in an instance, we provide the original group numbers on the line below the transformed instance for clarity (where a $d$ corresponds to an artificially inserted dummy group and $v_j$ denotes the $j^{\rm{th}}$ element in $V$). In this instance, two pairs of consecutive digits conflict, i.e., 2 conflicts with 3 and 6 conflicts with 7. Thus, we add 1 to every number from \(3\) to \(8\) and include \(3\) in a set \(V\), and we obtain 

\begin{table}[H]
\centering
\begin{tabular}{rrrrrrrrrr}
Transformed: & [2 & 5& 4& 7& 1& 9& 8& 6] &  V=\{3\} \\
Original: & [2 & 4& 3& 6& 1& 8& 7& 5] &
\end{tabular}
\end{table}

\noindent This process is repeated until no pair of consecutive digits with a conflict remains. In the updated sequence, the pair corresponding to the original digits 6 and 7 now appear as 7 and 8; thus, adding 1 to the numbers 8 and 9 and including 8 in \(V\) yields: 

\begin{table}[h!]
\centering
\begin{tabular}{rrrrrrrrrr}
Transformed: & [2 & 5& 4& 7& 1& 10 & 9& 6] &  V=\{3, 8\} \\
Original: & [2 & 4& 3& 6& 1& 8& 7& 5] &
\end{tabular}
\end{table}

\noindent Next, each element in \(V\) is appended immediately after the last digit in the sequence that is less than or equal to it, resulting in:

\begin{table}[H]
\centering
\begin{tabular}{rrrrrrrrrrr}
Transformed: & [ 2 & 5& 4& 7& 1& 3& 10 & 9& 6 & 8]  \\
Original: & [ 2 & 4& 3& 6& 1& $v_1$ & 8& 7& 5& $v_2$] 
\end{tabular}
\end{table}

\noindent In the above transformed sequence, group 2 should conflict with 5 (because the original sequence contains a conflict between 2 and 4 for ZSSS).  We therefore set \( \lambda = 1 \), increase the value of every digit greater than or equal to 1 by 1, and insert a dummy group $d$ between 3 and 6 (which were previously 2 and 5):

\begin{table}[H]
\centering
\begin{tabular}{rrrrrrrrrrrr}
Transformed: & [ 3 & $d$& 6& 5& 8& 2& 4& 11 & 10& 7 & 9]  \\
Original: & [ 2 & $d$ & 4& 3& 6& 1& $v_1$ & 8& 7& 5& $v_2$] 
\end{tabular}
\end{table}

\noindent Notice that the pair corresponding to the original digits 4 and 3 now appears as 6 and 5 in the transformed sequence. Because we require 4 and 3 to conflict with 6 in the original sequence, this implies that 6 and 5 should conflict with 8 in the above transformed sequence. Thus, we set \( \lambda = 2 \), and increase the value of every digit greater than or equal to 2 by 1. Next, we insert a dummy group $d$ between 6 and 9:

\begin{table}[H]
\centering
\begin{tabular}{rrrrrrrrrrrrr}
Transformed: & [ 4 & $d$& 7& 6& $d$& 9& 3& 5& 12 & 11& 8 & 10]  \\
Original: & [ 2 & $d$ & 4& 3& $d$ & 6& 1& $v_1$ & 8& 7& 5& $v_2$] 
\end{tabular}
\end{table}

\noindent Finally, we number the dummy groups from left to right in increasing order starting at 1, resulting in the sequence

\begin{table}[H]
\centering
\begin{tabular}{rrrrrrrrrrrrr}
Transformed: & [ 4 & 1& 7& 6& 2& 9& 3& 5& 12 & 11& 8 & 10]  \\
Original: & [ 2 & $d$ & 4& 3& $d$ & 6& 1& $v_1$ & 8& 7& 5& $v_2$] 
\end{tabular}
\end{table}

Figure \ref{RSP-sc_conflict_graph_case} displays the corresponding conflict graph under our modified conditions, where a conflict exists if and only if
all conditions in Table \ref{tab:conflict-conditions} hold simultaneously. This example demonstrates how the transformation alters both the node numbers and the associated conflict graph, but results in an equivalent $k$-coloring problem to the one corresponding to the SSP.

\begin{figure}[htbp!]
\begin{center}
\includegraphics[scale=0.5]{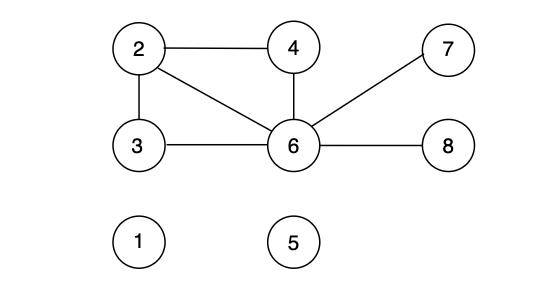}
\caption{Conflict graph for the example [2  4  3  6  1 8  7  5] from \textcite{KLstacksorting}.} \label{KLsequence}
\end{center}
\end{figure}

\begin{figure}[htbp!]
\begin{center}
\includegraphics[scale=0.24]{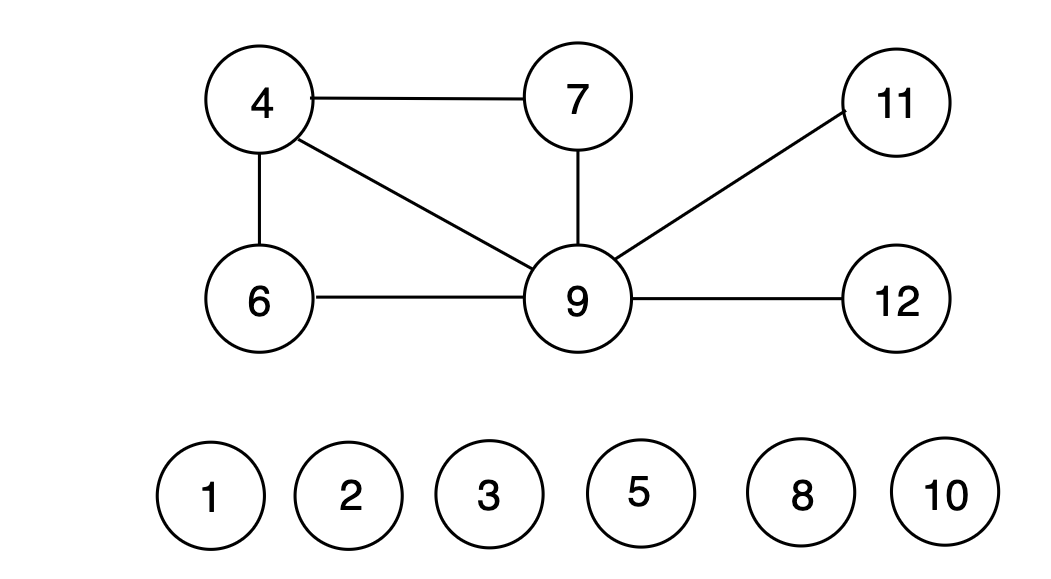}
\caption{Conflict graph for the RSP-sc based on the example from  \textcite{KLstacksorting}.} \label{RSP-sc_conflict_graph_case}
\end{center}
\end{figure}

Observe that the above transformation can be accomplished in polynomial time (requiring \(\mathcal{O}(n^2)\) steps), and no more than $n$ dummy nodes are added. We next show that there is a one-to-one mapping between the conflicts in the RSP-sc and conflicts in the SSP.  For convenience we will refer to the initial sequence in the equivalent SSP as the KL sequence.

\begin{lemma}\label{lem:dum}
Neither the dummy groups nor the artificial groups from the set $V$ conflict with any other groups.
\end{lemma}
\begin{proof}
For those dummy groups from the set $V$, which are inserted in Step 4 to permit creating a conflict between consecutive integers, because each such dummy group is inserted immediately after the last number in the sequence that is less than or equal to its value, Condition (ii) of Table \ref{tab:conflict-conditions} can never hold, either for smaller numbers preceding the dummy group or for higher numbers after it.  For those dummy groups from 1 through $\lambda$ created in Steps 5-9, these are sequenced in increasing order from left to right starting at 1, ensuring that Condition (ii) of Table \ref{tab:conflict-conditions} cannot hold for any of these groups.
\end{proof}

Note that the initial set of conflicts for the initial sequence $\sigma_A$ using the conflict conditions in Table \ref{tab:conflict-conditions} must be a subset of the conflicts defined for the ZSSS problem.  This holds because if Properties (i)-(ii) in Table \ref{tab:conflict-conditions} hold for a pair of groups, then the corresponding pair must have a conflict as defined for the ZSSS problem (i.e., $a_{i}<a_{i+v}$ and not all values smaller than $a_i$ appear before $a_{i+v}$).  

\begin{lemma}\label{lem:121}
A one-to-one correspondence exists between conflicts in the initial sequence $\sigma_A$ for the ZSSS problem and conflicts in the transformed sequence for the RSP-sc.  
\end{lemma}

\begin{proof} As noted earlier, the initial set of conflicts using the sequence $\sigma_A$ and Table \ref{tab:conflict-conditions} is a subset of those as defined for the ZSSS problem.  Let $\sigma_A'$ denote the transformed sequence.  By construction, each inserted dummy group creates a conflict in $\sigma_A'$ for the RSP-sc that corresponds to a conflict in $\sigma_A$ for ZSSS, while Lemma \ref{lem:dum} implies that the inserted dummy group does not conflict with any other groups. As a result, after the transformation, the set of conflicts in $\sigma_A$ for the ZSSS problem is a subset of those in $\sigma_A'$ for the RSP-sc.  To show that the set of conflicts in $\sigma_A'$ for the RSP-sc is also a subset of those in $\sigma_A$ for the ZSSS problem, suppose the transformation from $\sigma_A$ to $\sigma_A'$ creates a conflict between groups in $\sigma_A'$ for the RSP-sc where a conflict does not exist in $\sigma_A$ for ZSSS.  This implies that a pair of groups exists in $\sigma_A'$, $a_i'$ and $a_{i+v}'$, satisfying each of the properties in Table \ref{tab:conflict-conditions}.  In particular, by Conditions (i) and (ii), $a_i' < a_{i+v}'-1$ for some $v>0$, and not all numbers less than $a_i'$ appear before $a_{i+v}'$ in $\sigma_A'$, where $a_i'$ is the group number in $\sigma_A'$ corresponding to $a_i$ in $\sigma_A$. Suppose (for a contradiction) $a_{i+v}<a_i$ in $\sigma_A$; then $a_i'=a_i+n_1+n_2$ and $a_{i+v}'=a_{i+v}'+n_1$, where $n_1 \ge 0$ is the number of times both $a_i$ and $a_{i+v}$ were increased by 1 in the transformation, and $n_2\ge 0$ is the number of times $a_i$ was increased by 1 without $a_{i+v}$.  This gives $a_{i+v}'-a_i'=a_{i+v}-a_i-n_2$, so that $a_{i+v}-a_i>a_{i+v}'-a_i'>1$, a contradiction, implying that $a_{i+v}>a_i$ in $\sigma_A$.  Suppose next that not all numbers less than $a_i'$ appear before $a_{i+v}'$ in $\sigma_A'$.  Then at least one number $a_{i+v+k}'$ follows $a_{i+v}'$ with $k>0$ and $a_{i+v+k}'<a_i'$.  Suppose (again for a contradiction) $a_{i+v+k}>a_i$, which implies $a_{i+v+k}'=a_{i+v+k}+n_1+n_2$ and $a_i'=a_i+n_1$, where $n_1\ge 0$ is the number of times both $a_i$ and $a_{i+v+k}$ were increased by 1 together and $n_2\ge 0$ is the number of times $a_{i+v+k}$ was increased by 1 without $a_i$.  This implies $a_{i+v+k}'-a_{i+v+k}-n_2=a_i'-a_i$ or $a_{i+v+k}'-a_i'-n_2=a_{i+v+k}-a_i$, implying that $a_{i+v+k}'-a_i'\ge a_{i+v+k}-a_i$.  Because $a_{i+v+k}'-a_i'<0$, this implies $a_{i+v+k}-a_i<0$, a contradiction.  Therefore $a_{i+v+k}<a_i$, implying that $a_i$ and $a_{i+v}$ conflict in $\sigma_A$ for ZSSS.  As a result, any conflict in $\sigma_A'$ for the RSP-sc maps to a conflict in $\sigma_A$ for the ZSSS problem, and the set of conflicts in $\sigma_A'$ for the RSP-sc is therefore a subset of the set of conflicts in $\sigma_A$ for the ZSSS problem.  Because each conflict set is a subset of the other, the two conflict sets are identical.
\end{proof}

As a result, we have a polynomial-time transformation from an instance of the ZSSS problem to an equivalent instance of the RSP-sc, implying that a polynomial-time algorithm for the RSP-sc could be used to solve the ZSSS problem (with $m\ge 4$ stacks) in polynomial time, which is not possible unless $\mathcal{P}=\mathcal{NP}$.  This implies the result stated in Theorem \ref{th:NP}.

\subsection{Conflict Graph Properties}\label{sec:AppNPconflict}

Let $\sigma_A=[a_1 \ a_2 \ a_3 \ \cdots \ a_n]$ be an original sequence of all groups on a source track. We say that two groups are in \emph{conflict} if they cannot be placed on the same sorting track without requiring a shuffle move. We now describe the properties that characterize conflicts between groups. We distinguish two cases based on the relative values of \(a_i\) and \(a_{i+v}\).\\

\noindent \textbf{Case I: \(a_i < a_{i+v}-1\).} In this case, the following properties hold:

\begin{property}
\textbf{Same-Time, Current-Order Condition:} \(a_i\) and \(a_{i+v}\) can appear on the same sorting track at the same time in their current order (with \(a_i\) to the left of \(a_{i+v}\)) without any shuffles if and only if the entire subsequence $[a_i\ a_{i+1}\ \cdots\ a_{i+v-1}\ a_{i+v}]$ is moved together in one move, thereby preserving the original ordering.
\end{property}  

\begin{proof}
Suppose that groups \(a_i\) and \(a_{i+v}\) appear on the same sorting track at the same time in the same order as in $\sigma_A$ (with \(a_i\) to the left of \(a_{i+v}\)) without using any shuffles. This implies that the groups between \(a_i\) and \(a_{i+v}\) (i.e., \(a_{i+1}, a_{i+2}, \dots, a_{i+v-1}\)) must also appear on the track between \(a_i\) and \(a_{i+v}\) in the same order as in $\sigma_A$. (If some strict subsets of these groups had been moved in separate moves, the ordering in $\sigma_A$ would have been disrupted due to the FILO requirement, necessitating shuffles to recreate the ordering in $\sigma_A$.) Thus, the entire subsequence $[a_i\ a_{i+1}\ \cdots\ a_{i+v-1}\ a_{i+v}]$ must have been moved together in one move.

Conversely, assume that the entire subsequence $[a_i\ a_{i+1}\ \cdots\ a_{i+v-1}\ a_{i+v}]$ is moved together in one move. Then, by construction, the relative order of these groups is preserved exactly as in the initial sequence. Therefore, \(a_i\) and \(a_{i+v}\) (along with the intermediate groups) appear contiguously in their original order at the same time on the track without shuffles, which completes the proof.
\end{proof}

\begin{property}\label{prop:CDS_subsequence}
\textbf{CDS Structure Requirement:} If the subsequence $[a_i\ a_{i+1}\ \cdots\ a_{i+v}]$ does not consist of a sequence of CDSs (possibly of length 1) such that the first digit in each CDS is increasing, then \(a_i\) and \(a_{i+v}\) cannot appear on the same track at the same time in their current order without requiring a shuffle. Conversely, if the subsequence $[a_i\ a_{i+1}\ \cdots\ a_{i+v}]$ consists of a sequence of CDSs (possibly of length 1) such that the first digit in each CDS is increasing, then no conflict exists between $a_i$ and $a_{i+v}$.
\end{property} 

\begin{proof}
If the subsequence is not a sequence of CDSs such that the first digit in each CDS is increasing, then this implies that a subsequence exists such that a higher number $a_j$ appears before a lower number $a_{j+v}$ ($v>0$) that is not part of a CDS containing both numbers. For the sequence to maintain its current order, the entire subsequence must move to a sorting track in its current order. Because $a_j>a_{j+v}$ and the sequence $[a_j\ a_{j+1}\ \cdots\ a_{j+v}]$ is not a CDS, some subset of $[a_j\ a_{j+1}\ \cdots\ a_{j+v-1}]$ must be moved before $a_{j+v}$ can move to the sink.  Because $a_{j+v}$ must be at the sink before the subsequence $[a_j\ a_{j+1}\ \cdots\ a_{j+v-1}]$, this subsequence must move to another sorting stack before $a_j$ can be moved to the sink, indicating that at least one shuffle move is needed.  

If the subsequence $[a_i\ a_{i+1}\ \cdots\ a_{i+v}]$ consists of a sequence of CDSs (possibly of length 1) such that the first digit in each CDS is increasing, then if the entire subsequence is moved to a single track in one move, each individual CDS can subsequently move to the sink (after all digits smaller than the last group number in the CDS are in the sink) prior to the following CDS in the subsequence without requiring any shuffles.  
\end{proof}

\begin{property}
\textbf{Different-Time Condition:} \(a_i\) and \(a_{i+v}\) can be placed on the same track at different times only if all digits less than \(a_i\) appear in the subsequence $ [a_1\ a_2\ \cdots\ a_{i+v-1}]$.
\end{property} 

\begin{proof}
The only way $a_i$ and $a_{i+v}$ can appear on the same track at different times is if $a_i$ is moved to the track before $a_{i+v}$ and subsequently moved off the track before $a_{i+v}$ enters. This can only occur without a shuffle if all digits less than \(a_i\) appear in the subsequence $[a_1\ a_2\ \cdots\ a_{i+v-1}]$, which implies that a solution exists that moves all digits less than or equal to $a_i$ to the sink before moving \(a_{i+v}\) to a sorting track.
\end{proof}

\begin{property}\label{prop:consec_CDS}
\textbf{Same-Time, Reverse-Order Condition:} \(a_i\) and \(a_{i+v}\) can appear on the same track at the same time in reverse order (with \(a_{i+v}\) to the left of \(a_i\)) without shuffles if and only if a single CDS from \( a_{i+v} \) through $a_i$ can be created on a sorting track. This can only occur if the digits \(a_{i}+1,a_{i}+2, …, a_{i+v}-1, a_{i+v} \)
    appear in the initial ordering as a sequence of CDSs (possibly of length 1) such that the first digit in each CDS is increasing. 
\end{property}

\begin{proof}
Assume that \(a_i\) and \(a_{i+v}\) appear on the same sorting track at the same time in reverse order (with \(a_{i+v}\) to the left of \(a_i\)) without requiring any shuffles. Under this assumption, consider the subsequence from \(a_{i+v}\) through \(a_i\). There are two possibilities: either this subsequence forms a CDS or it does not. If it does not form a CDS, then since \(a_i < a_{i+v}\), \(a_i\) must be moved to the sink before \(a_{i+v}\). However, having \(a_{i+v}\) to the left of \(a_i\) blocks \(a_{i+v}\) from moving to the sink, which implies that a shuffle is necessary—a contradiction. Therefore, in order for \(a_i\) and \(a_{i+v}\) to be placed on the same track in reverse order without a shuffle, the subsequence \(a_{i+v}\) through \(a_i\) must indeed form a single CDS. This, in turn, requires that the digits \(a_i+1, a_i+2, \dots, a_{i+v}-1, a_{i+v}\) appear in $\sigma_A$ as a sequence of CDSs such that the first digit in each CDS is increasing.

Conversely, assume that the digits \(a_i+1, a_i+2, \dots, a_{i+v}-1, a_{i+v}\) appear in $\sigma_A$ as a sequence of CDSs (possibly of length 1) such that the first digit in each CDS is increasing. Under this assumption, the corresponding sequence of CDSs permits successively placing each such CDS to the left of the previous one in $\sigma_A$ on a sorting track, allowing for the creation of a single CDS  from \(a_{i+v}\) through \(a_i\) to be formed without any shuffles.
\end{proof}

\noindent \textbf{Case II: \(a_i > a_{i+v}\) and not all digits smaller than \(a_i\) appear before \(a_i\).}  
If $a_i>a_{i+v}$, then clearly $a_i$ and $a_{i+v}$ can use the same track without a shuffle if $a_i$ is moved to the track before $a_{i+v}$ (or if $[a_i\ \cdots\ a_{i+v}]$ is part of a single move to the track and where $[a_i\ \cdots\ a_{i+v}]$ is a subsequence of a single CDS). Thus, when \(a_i > a_{i+v}\), no conflict occurs between the two groups.


In summary, the above properties establish the following conditions under which a conflict occurs between groups in the RSP-sc.

\begin{proposition} \label{proposition:P1}
Consider an initial sequence \(\sigma_A=[a_1 \ a_2 \ a_3 \ \cdots \ a_n]\). Suppose that for some \(v>1\) we have \(a_{i+v}-1 > a_i\), and: 
\begin{enumerate}[i)]
    \item Not all digits less than \(a_i\) occur among the sequence \([a_1 \ a_2 \ \cdots \ a_{i+v-1}]\);
    \item The subsequence $[a_i\ a_{i+1}\ \cdots\ a_{i+v}]$ in \(\sigma_A\) does not consist of a sequence of CDSs (possibly of length 1) such that the first digit in each CDS is increasing; and
    \item The subsequence $[a_i\ a_i+1\ \cdots\ a_{i+v}-1\ a_{i+v}]$ of \(\sigma_A\) is not a sequence of CDSs (possibly of length 1) such that the first digit in each CDS is increasing.
\end{enumerate}
Then \(a_i\) and \(a_{i+v}\) cannot be assigned to the same stack without requiring at least one shuffle.
\end{proposition}

We define groups \( a_i \) and \( a_{i+v} \) (for \( v > 1 \)) as being in conflict if they satisfy Proposition \ref{proposition:P1}. The corresponding conflict graph for an initial sequence \(\sigma_A\) then contains a vertex corresponding to each group and an edge between any two vertices such that the corresponding groups are in conflict. 

\begin{proposition} \label{proposition:P2}
Consider an initial sequence \(\sigma_A=[a_1 \ a_2 \ \cdots \ a_n]\). For any \(a_i < n\) ($a_i>1$), \(a_i\) can never be in conflict with \(a_i+1\)  ($a_i-1$).
\end{proposition}

\begin{proof}
If $a_i$ appears in $\sigma_A$ after (before) $a_{i}+1$ ($a_i-1)$, then no conflict is created by moving $a_i+1$ ($a_i$) to a track and later moving $a_i$ ($a_i-1$) to the left of $a_i+1$ ($a_i$) on the same track. In either case, the smaller digit can move to the sink before the larger one without requiring a shuffle.  If $a_i$ appears in $\sigma_A$ before (after) $a_i+1$ ($a_i-1$), then the two groups form a sequence of successive CDSs (each of length one) from $a_i$ ($a_i-1$) to $a_i+1$ ($a_i$) and therefore satisfy the conditions of Property \ref{prop:consec_CDS}.  
\end{proof}

\begin{proposition} \label{proposition:P3}
Consider an initial sequence \(\sigma_A=[a_1 \ a_2 \ \cdots \ a_n]\). For every index \(i\) with \(1 \leq i < n\) ($1< i \le n$), \(a_i\) can never conflict with \(a_{i+1}\) ($a_{i-1}$).
\end{proposition}

\begin{proof} 
If $a_i>a_{i+1}$ ($a_i<a_{i-1}$), the larger number can move to a track before the smaller, in which case no conflict exists. Otherwise, the sequence of two groups forms a sequence of CDSs of length one, and by Property \ref{prop:CDS_subsequence}, no conflict exists between the two groups.
\end{proof}

\begin{proposition} \label{proposition:P4}
Consider an initial sequence \(\sigma_A=[a_1 \ a_2 \ \cdots \ a_n]\) with \(n \le 4\) and $k=1$ sorting track. In this case, no pair of groups \(a_i\) and \(a_{i+v}\) (for \( v > 1 \)) can be in conflict.
\end{proposition}

\begin{proof}
For each of the no more than $4!$ possible sequences, it is straightforward to demonstrate a zero-shuffle solution using a single sorting track.
\end{proof}

\section{Detailed Computational Test Results}
\label{sec:Computational_Details}

Table \ref{tab:30_mixed_results} presents detailed results on the computational performance of the MIP model and the ARG-DP algorithm across the 30 mixed problem instances (IDs 1-30). For the MIP model, the table includes the objective function value (``Obj.''), running time (``Time (s)'') measured in seconds, and time ratio. For the ARG-DP algorithm, it reports the objective function value (``Obj.''), running time (``Time (s)'') in seconds, and optimality gap expressed as a percentage (``Opt.\ Gap (\%)'').

\begin{table}[htbp]
    \centering 
    \footnotesize
    \caption{Computational results for the 30 mixed instances}
    \label{tab:30_mixed_results}
    \begin{tabular}{c*{6}{c}}
        \toprule
        & \multicolumn{3}{c}{MIP} & \multicolumn{3}{c}{ARG-DP} \\
        \cmidrule(lr){2-4} \cmidrule(lr){5-7}
        Test Instance ID & Obj. & Time (s) & Time Ratio & Obj. & Time (s) & Opt.\ Gap (\%) \\
        \midrule
         1  & 5.00 & 0.74   & 37.00  & 5.00 & 0.02  & 0.00 \\
         2  & 5.00 & 0.88   & 44.00  & 5.00 & 0.02  & 0.00 \\
         3  & 6.00 & 4.91   & 491.00 & 7.00 & 0.01  & 16.67 \\
         4  & 6.00 & 24.01  & 343.00 & 7.00 & 0.07  & 16.67 \\
         5  & 6.00 & 8.83   & 441.50 & 8.00 & 0.02  & 33.33 \\
         6  & 7.00 & 80.64  & 504.00 & 8.00 & 0.16  & 14.29 \\
         7  & 7.00 & 52.46  & 201.77 & 9.00 & 0.26  & 28.57 \\
         8  & 8.00 & 143.84 & 3.43   & 8.00 & 41.93 & 0.00 \\
         9  & 8.00 & 442.95 & 70.87  & 9.00 & 6.25  & 12.50 \\
        10  & 9.00 & 4774.88& 425.57 & 9.00 & 11.22 & 0.00 \\
        11  & 9.00 & 3881.51& 194.37 & 11.00& 19.97 & 22.22 \\
        12  & 5.00 & 1.20   & 40.00  & 5.00 & 0.03  & 0.00 \\
        13  & 6.00 & 9.51   & 3.35   & 6.00 & 2.84  & 0.00 \\
        14  & 7.00 & 24.49  & 8.97   & 7.00 & 2.73  & 0.00 \\
        15  & 7.00 & 16.91  & 1.79   & 8.00 & 9.43  & 14.29 \\
        16  & 7.00 & 90.42  & 291.68 & 8.00 & 0.31  & 14.29 \\
        17  & 7.00 & 291.20 & 12.87  & 9.00 & 22.62 & 28.57 \\
        18  & 8.00 & 224.54 & 42.77  & 8.00 & 5.25  & 0.00 \\
        19  & 6.00 & 49.68  & 9.48   & 6.00 & 5.24  & 0.00 \\
        20  & 8.00 & 205.29 & 662.23 & 8.00 & 0.31  & 0.00 \\
        21  & 8.00 & 1423.52& 241.68 & 8.00 & 5.89  & 0.00 \\
        22  & 9.00 & 2396.78& 23.04  & 9.00 & 104.01& 0.00 \\
        23  & 6.00 & 6.81   & 136.20 & 7.00 & 0.05  & 16.67 \\
        24  & 6.00 & 9.16   & 458.00 & 6.00 & 0.02  & 0.00 \\
        25  & 4.00 & 0.39   & 2.17   & 4.00 & 0.18  & 0.00 \\
        26  & 6.00 & 8.39   & 41.95  & 7.00 & 0.20  & 16.67 \\
        27  & 8.00 & 31.41  & 0.15   & 8.00 & 203.26& 0.00 \\
        28  & 8.00 & 24.09  & 52.37  & 9.00 & 0.46  & 12.50 \\
        29  & 9.00 & 221.53 & 15.46  & 9.00 & 14.33 & 0.00 \\
        30  & 8.00 & 124.05 & 2.83   & 8.00 & 43.85 & 0.00 \\
        \midrule
        \textbf{Overall Avg.} & 6.97 & 485.83 & 160.12 & 7.53 & 16.70 & 8.24 \\
        \bottomrule
    \end{tabular}
\end{table}

Similarly, Table \ref{tab:30_non_mixed_results} presents detailed results on the computational performance of the MIP model and the ARG-DP algorithm across the 30 non-mixed problem instances (IDs 31-60).

\begin{table}[htbp]
    \centering 
    \footnotesize
    \caption{Computational results for the 30 non-mixed instances}
    \label{tab:30_non_mixed_results}
    \begin{tabular}{c*{6}{c}}
        \toprule
        & \multicolumn{3}{c}{MIP} & \multicolumn{3}{c}{ARG-DP} \\
        \cmidrule(lr){2-4} \cmidrule(lr){5-7}
        Test Instance ID & Obj. & Time (s) & Time Ratio & Obj. & Time (s) & Opt.\ Gap (\%) \\
        \midrule
         31 & 5.00  & 1.61   & 53.67   & 5.00  & 0.03   & 0.00 \\
         32 & 4.00  & 0.43   & 107.50  & 4.00  & 0.004   & 0.00 \\
         33 & 6.00  & 3.48   & 43.50   & 6.00  & 0.08   & 0.00 \\
         34 & 4.00  & 0.05   & 25.00   & 4.00  & 0.002   & 0.00 \\
         35 & 4.00  & 0.33   & 66.00   & 4.00  & 0.005   & 0.00 \\
         36 & 5.00  & 0.08   & 8.00    & 5.00  & 0.01   & 0.00 \\
         37 & 7.00  & 10.17  & 113.00  & 7.00  & 0.09   & 0.00 \\
         38 & 8.00  & 507.05 & 272.61  & 9.00  & 1.86   & 12.50 \\
         39 & 6.00  & 3.36   & 336.00  & 7.00  & 0.01   & 16.67 \\
         40 & 9.00  & 50.04  & 5004.00 & 10.00 & 0.01   & 11.11 \\
         41 & 11.00 & 992.10 & 2917.94 & 12.00 & 0.34   & 9.09 \\
         42 & 10.00 & 61.20  & 6120.00 & 10.00 & 0.01   & 0.00 \\
         43 & 7.00  & 186.37 & 503.70  & 9.00  & 0.37   & 28.57 \\
         44 & 6.00  & 9.09   & 303.00  & 6.00  & 0.03   & 0.00 \\
         45 & 5.00  & 0.23   & 11.50   & 5.00  & 0.02   & 0.00 \\
         46 & 9.00  & 518.44 & 1.83    & 9.00  & 283.26 & 0.00 \\
         47 & 11.00 & 807.93 & 122.97  & 12.00 & 6.57   & 9.09 \\
         48 & 5.00  & 0.33   & 11.00   & 5.00  & 0.03   & 0.00 \\
         49 & 6.00  & 1.54   & 38.50   & 7.00  & 0.04   & 16.67 \\
         50 & 6.00  & 3.80   & 7.76    & 6.00  & 0.49   & 0.00 \\
         51 & 10.00 & 757.14 & 11.99   & 11.00 & 63.15  & 10.00 \\
         52 & 5.00  & 0.76   & 5.85    & 5.00  & 0.13   & 0.00 \\
         53 & 7.00  & 120.60 & 4.04    & 7.00  & 29.82  & 0.00 \\
         54 & 8.00  & 6.75   & 168.75  & 8.00  & 0.04   & 0.00 \\
         55 & 6.00  & 2.34   & 78.00   & 7.00  & 0.03   & 16.67 \\
         56 & 8.00  & 29.48  & 0.06    & 8.00  & 472.21 & 0.00 \\
         57 & 12.00 & 1130.67& 25.35   & 12.00 & 44.61  & 0.00 \\
         58 & 3.00  & 0.03   & 10.00   & 3.00  & 0.003   & 0.00 \\
         59 & 10.00 & 132.54 & 111.38  & 11.00 & 1.19   & 10.00 \\
         60 & 9.00  & 485.61 & 16.44   & 10.00 & 29.54  & 11.11 \\
        \midrule
        \textbf{Overall Avg.} & 7.07  & 194.12 & 549.98  & 7.47  & 31.13  & 5.05 \\
        \bottomrule
    \end{tabular}
\end{table}

Tables \ref{tab:5_mixed_results} and \ref{tab:5_non_mixed_results}  present detailed results on the computational performance of the MIP model and the ARG-DP algorithm for the five Gaia-mixed cases (IDs 61-65) and the five Gaia-non-mixed cases (IDs 66-70).

\begin{table}[htbp]
    \centering 
    \footnotesize
    \caption{Computational results for the five Gaia-mixed cases}
    \label{tab:5_mixed_results}
    \begin{tabular}{c*{6}{c}}
        \toprule
        & \multicolumn{3}{c}{MIP} & \multicolumn{3}{c}{ARG-DP} \\
        \cmidrule(lr){2-4} \cmidrule(lr){5-7}
        Test Case ID & Obj. & Time (s) & Time Ratio & Obj. & Time (s) & Opt.\ Gap (\%) \\
        \midrule
         61 & 13.00 & 154.53  & 1.20   & 14.00 & 128.92 & 7.69  \\
         62 &  9.00 & 3.92    & 15.08  &  9.00 & 0.26   & 0.00  \\
         63 &  8.00 & 20.88   & 1044.00&  8.00 & 0.02   & 0.00  \\
         64 & 12.00 & 6477.66 & 239.91 & 14.00 & 27.00  & 16.67 \\
         65 &  8.00 & 53.37   & 102.63 &  8.00 & 0.52   & 0.00  \\
        \midrule
        \textbf{Overall Avg.} & 10.00 & 1342.07 & 280.56 & 10.60 & 31.34 & 4.87 \\
        \bottomrule
    \end{tabular}
\end{table}

\begin{table}[htbp]
    \centering 
    \footnotesize
    \caption{Computational results for the five Gaia-non-mixed cases}
    \label{tab:5_non_mixed_results}
    \begin{tabular}{c*{6}{c}}
        \toprule
        & \multicolumn{3}{c}{MIP} & \multicolumn{3}{c}{ARG-DP} \\
        \cmidrule(lr){2-4} \cmidrule(lr){5-7}
        Test Case ID & Obj. & Time (s) & Time Ratio & Obj. & Time (s) & Opt.\ Gap (\%) \\
        \midrule
         66 & 11.00 & 14181.43 & 768.64  & 11.00 & 18.45  & 0.00 \\
         67 & 11.00 & 0.75     & 37.50   & 11.00 & 0.02   & 0.00 \\
         68 & 16.00 & 17653.62 & 19.85   & 20.00 & 889.20 & 25.00 \\
         69 & 14.00 & 50932.42 & 42.16   & 15.00 & 1208.01& 7.14 \\
         70 & 8.00  & 5.16     & 19.11   & 9.00  & 0.27   & 12.50 \\
        \midrule
        \textbf{Overall Avg.} & 12.00 & 16554.68 & 177.45 & 13.20 & 423.19 & 8.93 \\
        \bottomrule
    \end{tabular}
\end{table}

\end{appendices}

\end{document}